\documentclass[fontsize=11pt]{scrartcl}

\usepackage{amsmath,amssymb,amsfonts}
\usepackage{amsthm} 

\usepackage{mathrsfs}
\usepackage[round]{natbib}
\usepackage{enumerate}
\usepackage{placeins} 
\usepackage[pdftex]{graphicx}
\usepackage{subcaption}
\usepackage{theoremref} 
\usepackage{float} 
\usepackage{nccfoots}

\newtheorem{theorem}{Theorem}[section]
\newtheorem{definition}{Definition}[section]
\newtheorem{lemma}{Lemma}[section]
\newtheorem{prop}{Proposition}[section]
\newtheorem{coro}{Corollary}[section]
\theoremstyle{definition} 
\newtheorem{assumption}{Assumption}[section]
\theoremstyle{remark} 

\newtheorem{example}{Example}[section]

\DeclareMathOperator{\E}{E}

\DeclareMathOperator{\Cov}{Cov}
\DeclareMathOperator{\F}{F}
\DeclareMathOperator{\p}{P}

\DeclareMathOperator{\NIID}{NIID}
\DeclareMathOperator{\argmax}{argmax}

\begin{document}

\title{Robust Estimation of Change-Point Location}
\author{Carina Gerstenberger$^*$}
\date{}


\maketitle

\begin{abstract}
\footnotesize
We introduce a robust estimator of the location parameter for the change-point in the mean based on the Wilcoxon statistic and establish its consistency for $L_1$ near epoch dependent processes. It is shown that the consistency rate depends on the magnitude of change. A simulation study is performed to evaluate finite sample properties of the Wilcoxon-type estimator in standard cases, as well as under heavy-tailed distributions and disturbances by outliers, and to compare it with a CUSUM-type estimator. It shows that the Wilcoxon-type estimator is equivalent to the CUSUM-type estimator in standard cases, but outperforms the CUSUM-type estimator in presence of heavy tails or outliers in the data.\par 
KEYWORDS: Wilcoxon statistic; change-point estimator; near epoch dependence \Footnotetext{}{Date: January 9, 2017.}\Footnotetext{*}{Fakult\"at f\"ur Mathematik, Ruhr-Universit\"at Bo\-chum, 
44780 Bochum, Germany}
\end{abstract}

\section{Introduction}\label{Introduction}

In many applications it can not be assumed that observed data have a constant mean over time. Therefore, extensive research has been done in testing for change-points in the mean, see e.g. \citet{Giraitis.1996}, \citet{Csorgo.1997}, \citet{Ling.2007}, and others. A number of papers deal with the problem of estimation of the change-point location. \citet{Bai.1993} estimates the unknown location point for the break in the mean of a linear process by the method of least squares. \citet{Antoch.1995} and \citet{Csorgo.1997} established the consistency rates for CUSUM-type estimators for independent data, while \citet{Csorgo.1997} considered weakly dependent variables. \citet{Horvath.1997} established consistency of CUSUM-type estimators of location of change-point for strongly dependent variables. \citet{Kokoszka.1998,Kokoszka.2000} discussed CUSUM-type estimators for dependent observations and ARCH models. In spite of numerous studies on testing for changes and estimating for change-points, however, just a few procedures are robust against outliers in the data. In a recent work \citet{Dehling.2015} address the robustness problem of testing for change-points by introducing a Wilcoxon-type test which is applicable under short-range dependence (see also \citet{Dehling.2013b} for the long-range dependence case).\par 
In this paper we suggest a robust Wilcoxon-type estimator for the change-point location based on the idea of \citet{Dehling.2015} and applicable for $L_1$ near epoch dependent processes. The Wilcoxon change-point test statistic is defined as
\begin{equation}\label{Wilcoxon-TS}
W_{n}(k) = \sum_{i=1}^{k}\sum_{j=k+1}^{n}(1_{\{X_i\leq X_j\}}-1/2)
\end{equation}
and counts how often an observation of the second part of the sample, $X_{k+1},\ldots,X_n$, exceeds an observation of the first part, $X_1,\ldots,X_k$. Assuming a change in mean happens at the time $k^*$, the absolute value of $W_{n}(k^*)$ is expected to be large. Hence, the Wilcoxon-type estimator for the location of the change-point, 
\begin{equation}\label{bp_estimator_wilcoxon}
\hat{k} = \min\Big\{k:\max_{1\leq l < n}\Big| W_{n}(l) \Big|=\Big| W_{n}(k) \Big|\Big\},
\end{equation}
can be defined as the smallest k for which the Wilcoxon test statistic $W_n(k)$ attains its maximum. Since the Wilcoxon test statistic is a rank-type statistic, outliers in the observed data can not affect the test statistic significantly. On the contrary, the CUSUM-type test statistic
\begin{equation*}
C_{n}(k) = \frac{1}{k}\sum_{i=1}^{k} X_i - \frac{1}{n}\sum_{i=1}^{n} X_i,
\end{equation*}
which compares the difference of the sample mean of the first $k$ observations and the sample mean over all observations, can be significantly disturbed by a single outlier.\par 
The outline of the paper is as follows. In Section \ref{Main Results} we discuss the consistency and the rates of the estimator $\hat{k}$ in (\ref{bp_estimator_wilcoxon}). Section \ref{Simulation Study} contains the simulation study. Section \ref{Useful Properties} provides useful properties of the Wilcoxon test statistic and the proof of the main result. Sections \ref{Auxiliary Results} and \ref{result_literature} contain some auxiliary results.

\section{Definitions, assumptions and main results}\label{Main Results}

Assume the random variables $X_1,\ldots,X_n$ follow the change-point model 
\begin{equation}\label{model_bp}
X_i = \begin{cases}
Y_i + \mu, & 1\leq i \leq k^*\\
Y_i + \mu + \Delta_n, & k^* < i \leq n,
\end{cases}
\end{equation}
where the process $(Y_j)$ is a stationary zero mean short-range dependent process, $k^*$ denotes the location of the unknown change-point and $\mu$ and $\mu + \Delta_n$ are the unknown means. We assume that $Y_1$ has a continuous distribution function $\F$ with bounded second derivative and that the distribution functions of $Y_1-Y_k$, $k\geq 1$ satisfy
\begin{equation}\label{Bed_gem_Verteilung}
\p(x\leq Y_1-Y_k\leq y) \leq C|y-x|,
\end{equation}
for all $0\leq x\leq y \leq 1$, where $C$ does not depend on $k$ and $x,y$. We allow the magnitude of the change $\Delta_n$ vary with the sample size $n$.

\begin{assumption}\label{assumption_change}
\begin{enumerate}[a)]
\item The change-point $k^*=[n\theta]$, $0<\theta<1,$ is proportional to the sample size $n$. 
\item The magnitude of change $\Delta_n$ depends on the sample size $n$, and is such that
\begin{equation} \label{assumption_size_change_infty}
\Delta_n \rightarrow 0, \qquad n\Delta_n^2 \rightarrow \infty, \qquad n\rightarrow \infty.
\end{equation}
\end{enumerate}
\end{assumption}
\par \bigskip
Next we specify the assumptions on the underlying process $(Y_j)$. The following definition introduces the concept of an absolutely regular process which is also known as $\beta$-mixing.

\begin{definition}
A stationary process $\left(Z_j\right)_{j\in\mathbb{Z}}$ is called  absolutely regular if
\[
\beta_k = \sup_{n\geq 1}\E \sup_{A\in \mathcal{F}_1^n} \left| \p\left(A|\mathcal{F}_{n+k}^{\infty}\right)-\p\left(A\right)\right| \rightarrow 0
\]
as $k\rightarrow\infty$, where $\mathcal{F}_a^b$ is the $\sigma$-field generated by random variables $Z_a,\ldots,Z_b$.
\end{definition}

The coefficients $\beta_k$ are called mixing coefficients. For further information about mixing conditions see \citet{Bradley.2002}. The concept of absolute regularity covers a wide range of processes. However, important processes like linear processes or AR processes might not be absolutely regular. To overcome this restriction, in this paper we discuss functionals of absolutely regular processes, i.e. instead of focusing on the absolute regular process $\left(Z_j\right)$ itself, we consider process $\left(Y_j\right)$ with $Y_j=f(Z_j,Z_{j-1},Z_{j-2},\ldots)$, where $f:\mathbb{R}^{\mathbb{Z}}\rightarrow\mathbb{R}$ is a measurable function. The following near epoch dependence condition ensures that $Y_j$ mainly depends on the near past of $\left(Z_j\right)$.

\begin{definition}\thlabel{NED}
We say that stationary process $\left(Y_j\right)$ is $L_1$ near epoch dependent ($L_1$ NED) on a stationary process $\left(Z_j\right)$ with approximation constants $a_k$, $k\geq 0$, if conditional expectations $\E(Y_1|\mathcal{G}_{-k}^k)$, where $\mathcal{G}_{-k}^{k}$ is the $\sigma$-field generated by $Z_{-k},\ldots,Z_{k}$, have property
\[
\E\Big| Y_1 - \E(Y_1|\mathcal{G}_{-k}^k)\Big| \leq a_k, \qquad k=0,1,2,\ldots
\]
and $a_k\rightarrow0$, $k\rightarrow\infty$.
\end{definition}

Note that $L_1$ NED is a special case of more general $L_r$ near epoch dependence, where approximation constants are defined using $L_r$ norm: $\E\big| Y_1 - \E(Y_1|\mathcal{G}_{-k}^k)\big|^r \leq a_k$, $r\geq 1$. $L_r$ NED processes are also called $r$-approximating functionals. In testing problems considered in this paper we allow for heavy-tailed distributions. Hence, we deal with $L_1$ near epoch dependence, which assumes existence of only the first moment $\E|Y_1|$. The concept of near epoch dependence is applicable e.g. to GARCH(1,1) processes, see \citet{Hansen.1991}, and linear processes, see \thref{example_linear_processes} below. \citet{borovkova.2001} provide additional examples and information about properties of $L_r$ near epoch dependent process.\par 

\begin{example}\thlabel{example_linear_processes}
Let $\left(Y_j\right)$ be a linear process, i.e. $Y_t=\sum_{j=0}^{\infty}\psi_j Z_{t-j}$, where $\left(Z_j\right)$ is white-noise process and the coefficients $\psi_j$, $j\geq 0$, are absolutely summable. Since $\left(Z_j\right)$ is stationary and $Z_{t-j}$ is $\mathcal{G}_{-k}^k$ measurable for $|t-j|\leq k$, we get
\begin{multline*}
\E|Y_t-\E(Y_t|\mathcal{G}_{-k}^k)| \leq  \sum_{j=k+1}^{\infty}|\psi_j|\E|Z_{t-j}-\E(Z_{t-j}|\mathcal{G}_{-k}^k)|\\
\leq 2\sum_{j=k+1}^{\infty}|\psi_j|\E|Z_{t-j}| = 2\E|Z_1|\sum_{j=k+1}^{\infty}|\psi_j|.
\end{multline*}
Thus, the linear process $\left(Y_j\right)$ is $L_1$ NED on $\left(Z_j\right)$ with approximation constants $a_k=2\E|Z_1|\sum_{j=k+1}^{\infty}|\psi_j|$.
\end{example}
\par \bigskip

We will assume that the process $(Y_j)$ in (\ref{model_bp}) is $L_1$ near epoch dependent on some absolutely regular process $\left(Z_j\right)$. In addition, we impose the following condition on the decay of the mixing coefficients $\beta_k$ and approximation constants $a_k$:
\begin{equation}\label{condition_appr.const_regu.coeff}
\sum_{k=1}^{\infty}{k^2\left(\beta_k+\sqrt{a_k}\right)}< \infty.
\end{equation}

The next theorem states the rates of consistency of the Wilcoxon-type change-point estimator $\hat{k}$ given in (\ref{bp_estimator_wilcoxon}) and the estimator $\hat{\theta}=\hat{k}/n$ of the true location parameter $\theta$ for the change-point $k^*=[n\theta]$.

\begin{theorem}\thlabel{consistence_of_bp_estimator}
Let $X_1,\ldots, X_n$ follow the change-point model (\ref{model_bp}) and Assumption \ref{assumption_change} be satisfied. Assume that $\left(Y_j\right)$ is a stationary zero mean $L_1$ near epoch dependent process on some absolutely regular process $\left(Z_j\right)$ and (\ref{condition_appr.const_regu.coeff}) holds. Then,
\begin{equation}\label{rate_consistency_k_hat}
\Big| \hat{k}-k^*\Big| = O_P\Big(\frac{1}{\Delta_n^2}\Big),
\end{equation}
and
\begin{equation}\label{rate_consistency_theta}
\Big|\hat{\theta}-\theta\Big| =O_P\Big(\frac{1}{n\Delta_n^2}\Big).
\end{equation}
\end{theorem}
\par \bigskip
The rate of consistency of $\hat{\theta}$ in (\ref{rate_consistency_theta}) is given by $n\Delta_n^2$. The assumption $n\Delta_n^2\rightarrow\infty$ in (\ref{assumption_size_change_infty}) implies $\hat{k}-k^*=o_P(k^*)$ and yields consistency of the estimator: $\hat{\theta}\rightarrow_p\theta$. In particular, for $\Delta_n\geq n^{-1/2+\epsilon}$, $\epsilon>0$, the rate of consistency in (\ref{rate_consistency_theta}) is $n^{2\epsilon}$: $\big|\hat{\theta}-\theta\big|= O_P\left(n^{-2\epsilon}\right)$.\par 
The same consistency rate $n^{\epsilon}$ for the CUSUM-type change-point location estimator $\tilde{\theta}_C = \tilde{k}_C/n$, given by
\begin{equation}\label{bp_estimator_cusum}
\tilde{k}_C  = \min\bigg\{ k: \max_{1\leq i \leq n}\bigg| \sum_{j=1}^{i}X_j - \frac{i}{n}\sum_{j=1}^{n}X_j \bigg| =  \bigg| \sum_{j=1}^{k}X_j - \frac{k}{n}\sum_{j=1}^{n}X_j \bigg|  \bigg\},
\end{equation}
was established by \citet{Antoch.1995} for independent data and by \citet{Csorgo.1997} for weakly dependent data.

\section{Simulation results}\label{Simulation Study}

In this simulation study we compare the finite sample properties of the Wilcoxon-type change-point estimator $\hat{k}$, given in (\ref{bp_estimator_wilcoxon}), with the CUSUM-type estimator $\tilde{k}_C$, given in (\ref{bp_estimator_cusum}).
We refer to the Wilcoxon-type change-point estimator by W and to the CUSUM-type estimator by C.\par 

We generate the sample of random variables $X_1,\ldots,X_n$ using the model
\begin{equation}\label{model_bp_simu}
X_i = \begin{cases}
Y_i + \mu &, 1\leq i \leq k^*\\
Y_i + \mu + \Delta &, k^* < i \leq n
\end{cases}
\end{equation}
where $Y_i = \rho Y_{i-1} + \epsilon_i$ is an AR(1) process. In our simulations we consider $\rho = 0.4$, which yields a moderate positive autocorrelation in $X_i$. The innovations $\epsilon_i$ are generated from a standard normal distribution and a Student's t-distribution with 1 degree of freedom. We consider the time of change $k^*=[n\theta]$, $\theta = 0.25,0.5,0.75$, the magnitude of change $\Delta = 0.5,1,2$ and the sample sizes $n=50,100,200,500$. All simulation results are based on 10.000 replications. Note that we report estimation results not for $\hat{k}$ and $\tilde{k}_C$, but $\hat{\theta}=\hat{k}/n$ and $\tilde{\theta}_C=\tilde{k}_C/n$.\par 
Figure \ref{figure.normal.hist} contains the histogram based on the sample of 10.000 values of Wilcoxon-type estimator $\hat{\theta}$ and the CUSUM-type estimator $\tilde{\theta}_C$, for the model (\ref{model_bp_simu}) with $\Delta=1$, $\theta=0.5$, $n=50$ and independent standard normal innovations $\epsilon_i$. Both estimation methods give very similar histograms.\par 
Table \ref{table.normal} reports the sample mean and the sample standard deviation based on 10.000 values of $\hat{\theta}$ and $\tilde{\theta}_C$ for other choices of parameters $\Delta$ and $\theta$. It shows that performance of both estimators improves when the sample size $n$ and the magnitude of change $\Delta$ are rising, and when the change happens in the middle of the sample. In general, Wilcoxon-type estimator performs in all experiments as good as the CUSUM-type estimator.\par 
Figure \ref{figure.t1.hist} shows the histogram based on 10.000 values of $\hat{\theta}$ and $\tilde{\theta}_C$, for the model (\ref{model_bp_simu}) with $t_1$-distributed heavy-tailed iid innovations $\epsilon_i$, $\Delta=1$, $\theta=0.5$ and $n=500$. For heavy-tailed innovations $\epsilon_i$, both estimators deviate from the true value of the parameter $\theta$ more significantly than under normal innovations. Nevertheless, the Wilcoxon-type estimator seems to outperform the CUSUM-type estimator.\par 
Figure \ref{figure.outliers.hist} shows the histogram based on 10.000 values for $\hat{\theta}$ and $\tilde{\theta}_C$ when the data $X_1,\ldots,X_n$ is generated by (\ref{model_bp_simu}) with $\Delta =1$, $\theta=0.5$, $n=200$ and $\epsilon_i \sim \NIID(0,1)$ and contains outliers. The outliers are introduced by multiplying observations $X_{[0.2n]}$, $X_{[0.3n]}$, $X_{[0.6n]}$ and $X_{[0.8n]}$ by the constant $M=50$. The histogram shows that the Wilcoxon-type estimator is rarely affected by the outliers, whereas the CUSUM-type estimator suffers large distortions.\par 
Table \ref{table.normal.t1.outliers} reports the sample mean and the sample standard deviation based on 10.000 values of $\hat{\theta}$ and $\tilde{\theta}_C$ for $\Delta=1$ and $\theta=0.5$ for sample size $n=50,100,200,500$ in the case of the normal, normal with outliers and $t_1$-distributed innovations. Figures \ref{figure.normal.hist}, \ref{figure.t1.hist} and \ref{figure.outliers.hist} presents results for $n=50,200,500$.\par 
In general, we conclude that the Wilcoxon-type change-point location estimator performs equally well as the CUSUM-type change-point estimator in standard situations, but outperforms the CUSUM-type estimator in presence of heavy tails and outliers.

 \begin{figure}
 \begin{subfigure}[c]{0.5\textwidth}
\includegraphics[width=7cm, height=5cm]{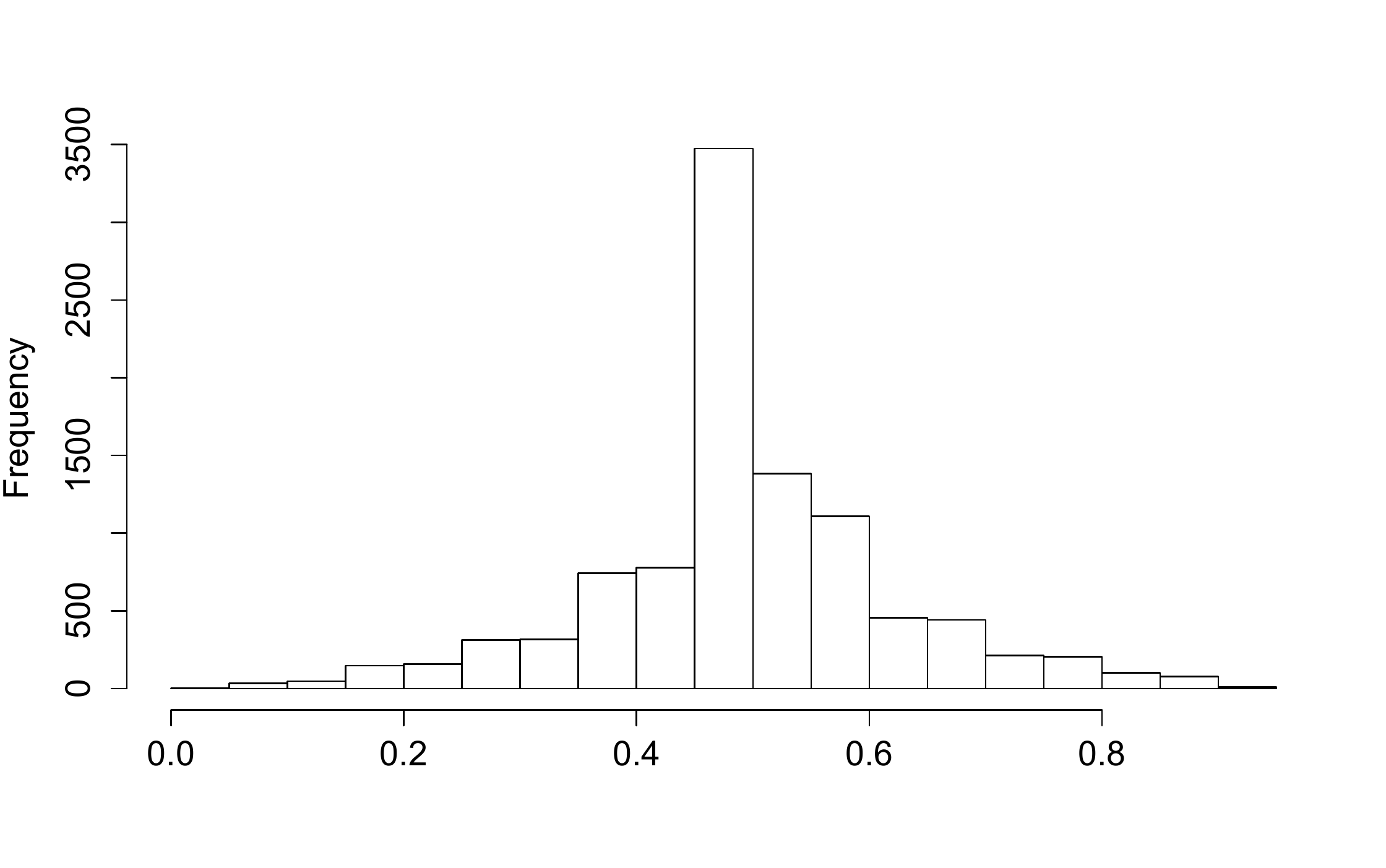}
\subcaption{CUSUM $\tilde{\theta}_C$} 
\end{subfigure}
 \begin{subfigure}[c]{0.5\textwidth}
\includegraphics[width=7cm, height=5cm]{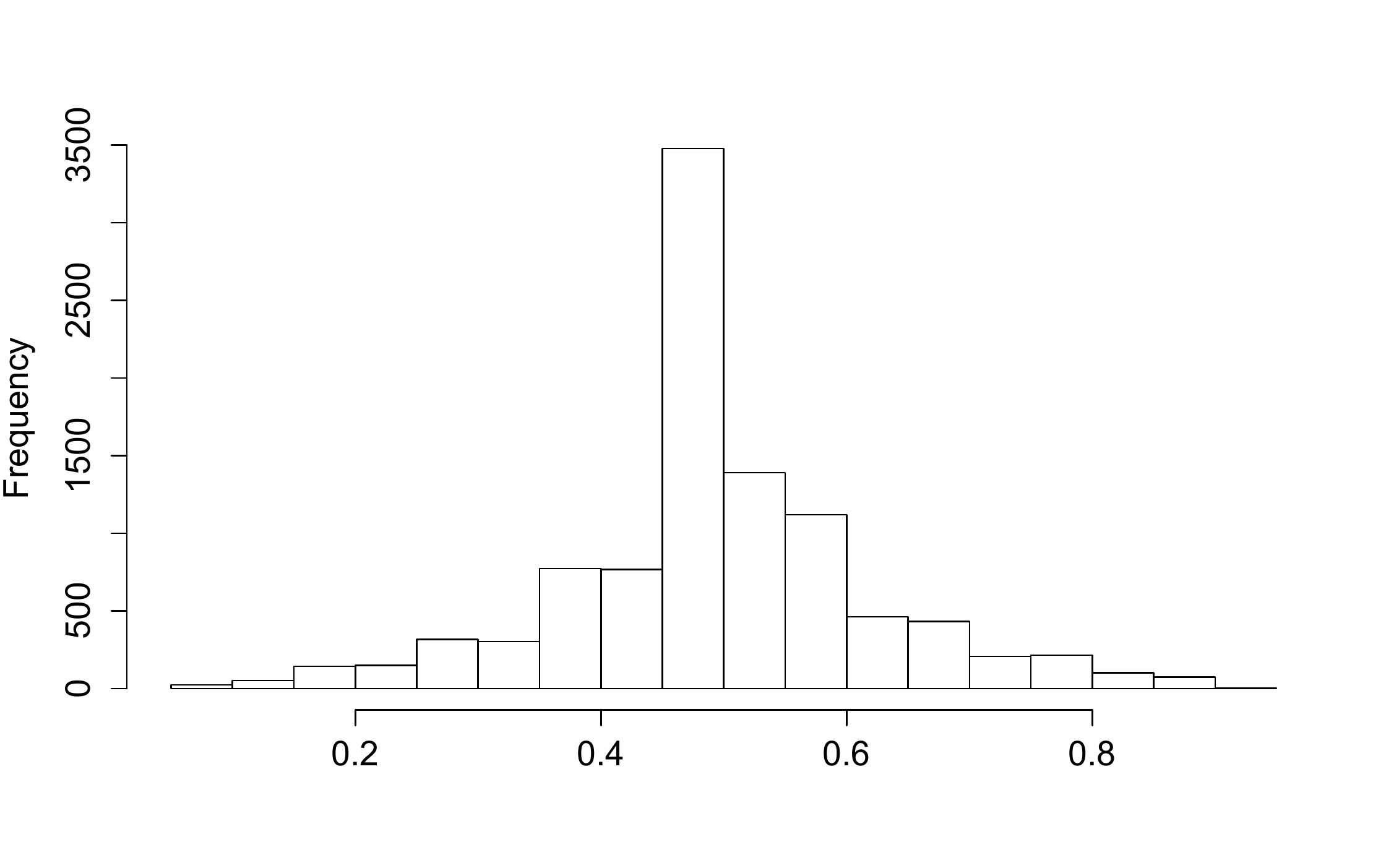}
\subcaption{Wilcoxon $\hat{\theta}$} 
\end{subfigure}
\caption{Histogram based on 10.000 values for the Wilcoxon-type estimator $\hat{\theta}$ and the CUSUM-type estimator $\tilde{\theta}_C$. $X_i$ follows the model (\ref{model_bp_simu}) with $\Delta=1$, $\theta=0.5$, $n=50$ and normal innovations $\epsilon_i \sim \NIID(0,1)$.}\label{figure.normal.hist}
\end{figure}

\begin{table}
\begin{tabular}{c|c|c|r|r|r|r|r|r|r|r}
\multicolumn{3}{c|}{} & \multicolumn{2}{|c|}{n=50} & \multicolumn{2}{|c|}{n=100} & \multicolumn{2}{|c|}{n=200}  & \multicolumn{2}{|c}{n=500}\\\hline
 $\Delta$ & $\theta$ &  & \multicolumn{1}{|c|}{C} & \multicolumn{1}{|c|}{W}& \multicolumn{1}{|c|}{C} & \multicolumn{1}{|c}{W} & \multicolumn{1}{|c|}{C} & \multicolumn{1}{|c}{W}& \multicolumn{1}{|c|}{C} & \multicolumn{1}{|c}{W}\\\hline
0.5	& 0.25 	& mean	&0.46 &0.46	&0.43 &0.44	&0.40 &0.40 &0.34 &0.34\\
	& 		& sd	&0.21 &0.21	&0.20 &0.20	&0.18 &0.18 &0.13 &0.13\\
	& 0.50 	& mean 	&0.50 &0.50	&0.50 &0.50	&0.50 &0.50 &0.50 &0.50\\
	& 		& sd	&0.18 &0.18	&0.16 &0.16	&0.13 &0.13 &0.08 &0.08\\
 	& 0.75 	& mean 	&0.54 &0.54	&0.57 &0.56	&0.61 &0.61 &0.66 &0.66\\
	& 		& sd	&0.20 &0.20	&0.20 &0.20	&0.18 &0.18 &0.13 &0.13\\
1 	& 0.25 	& mean	&0.39 &0.39	&0.35 &0.35	&0.31 &0.31 &0.28 &0.28\\
	& 		& sd	&0.18 &0.18	&0.14 &0.14	&0.10 &0.10 &0.05 &0.06\\
	& 0.50 	& mean 	&0.50 &0.50	&0.50 &0.50	&0.50 &0.50 &0.50 &0.50\\
	& 		& sd	&0.12 &0.12	&0.09 &0.09	&0.05 &0.05 &0.02 &0.02\\
  	& 0.75 	& mean 	&0.61 &0.60	&0.65 &0.65	&0.69 &0.69 &0.72 &0.72\\
	& 		& sd	&0.17 &0.17	&0.15 &0.15	&0.10 &0.10 &0.05 &0.06\\
2 	& 0.25 	& mean	&0.30 &0.31	&0.28 &0.29	&0.27 &0.28 &0.26 &0.26\\
	& 		& sd	&0.10 &0.10	&0.06 &0.07	&0.04 &0.04 &0.02 &0.02\\
	& 0.50 	& mean 	&0.50 &0.50	&0.50 &0.50	&0.50 &0.50 &0.50 &0.50\\
	& 		& sd	&0.05 &0.05	&0.03 &0.03	&0.02 &0.01 &0.01 &0.01\\
  	& 0.75 	& mean 	&0.69 &0.68	&0.72 &0.71	&0.73 &0.73 &0.74 &0.74\\
	& 		& sd	&0.09 &0.10	&0.06 &0.07	&0.04 &0.04 &0.02 &0.02\\
\end{tabular}
\caption{Sample mean and the sample standard deviation based on 10.000 values of $\hat{\theta}$ and $\tilde{\theta}_C$. $X_i$ follows the model (\ref{model_bp_simu}) with normal innovations $\epsilon_i \sim \NIID(0,1)$.}\label{table.normal}
\end{table} 

 \begin{figure}
 \begin{subfigure}[c]{0.5\textwidth}
\includegraphics[width=7cm, height=5cm]{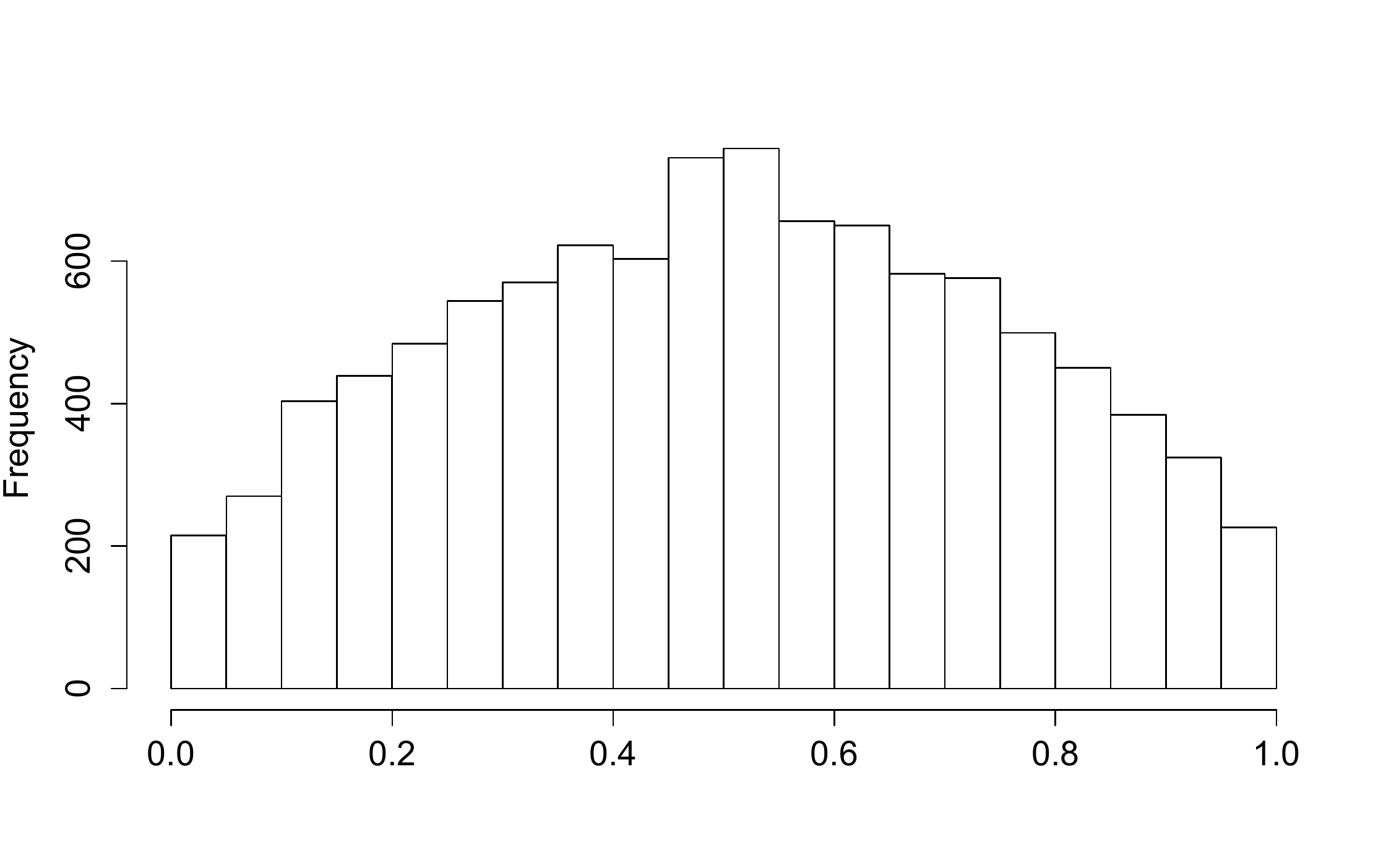}
\subcaption{CUSUM $\tilde{\theta}_C$} 
\end{subfigure}
 \begin{subfigure}[c]{0.5\textwidth}
\includegraphics[width=7cm, height=5cm]{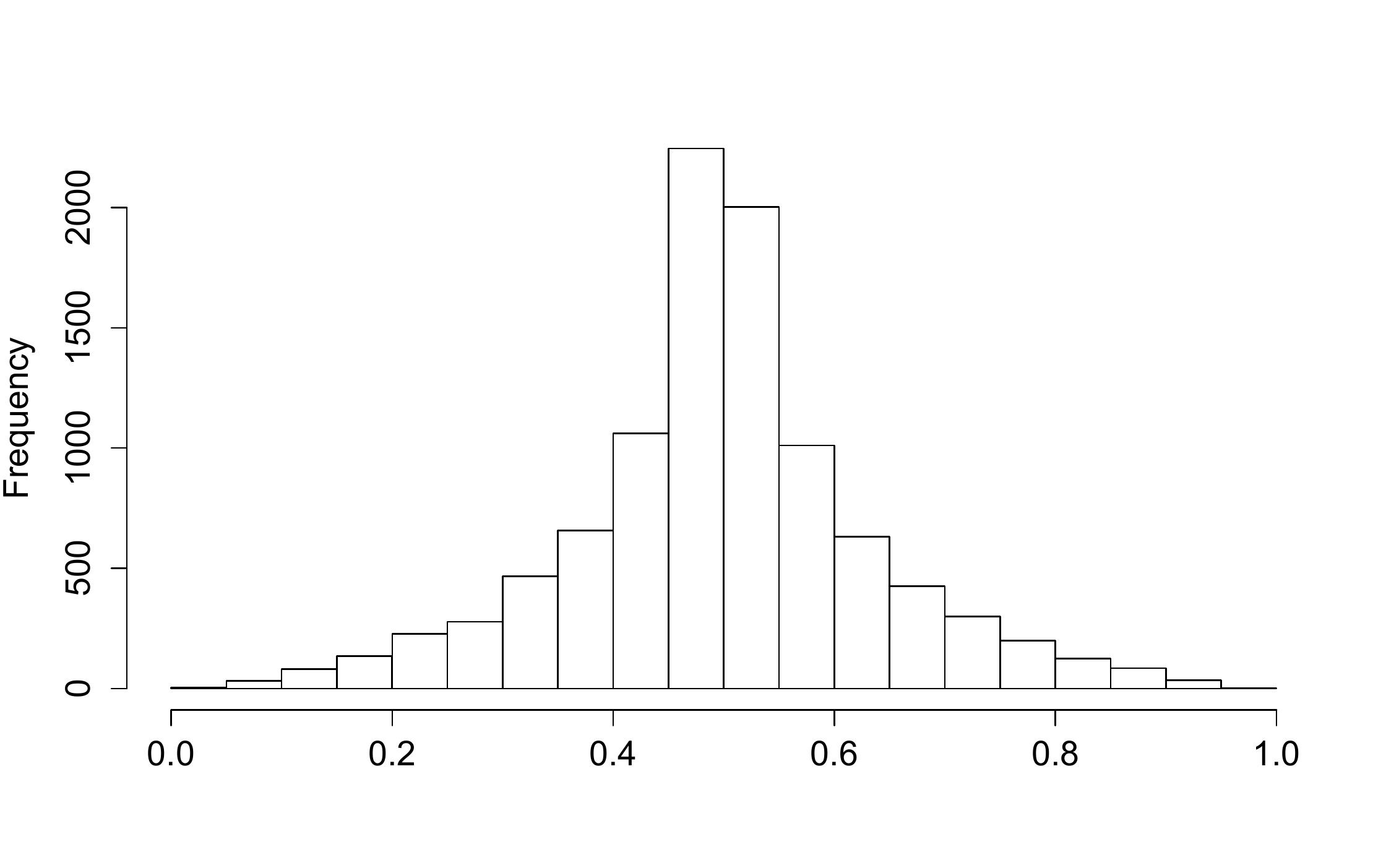}
\subcaption{Wilcoxon $\hat{\theta}$} 
\end{subfigure}
\caption{Histogram of CUSUM-type estimator $\tilde{\theta}_C$ and Wilcoxon-type estimator $\hat{\theta}$ based on 10.000 values of $\tilde{\theta}_C$ and $\hat{\theta}$ for the model (\ref{model_bp_simu}) with iid $t_1$-distributed innovations, $\Delta=1$, $\theta=0.5$ and $n=500$.}\label{figure.t1.hist}
\end{figure}

\begin{figure}
 \begin{subfigure}[c]{0.5\textwidth}
\includegraphics[width=7cm, height=5cm]{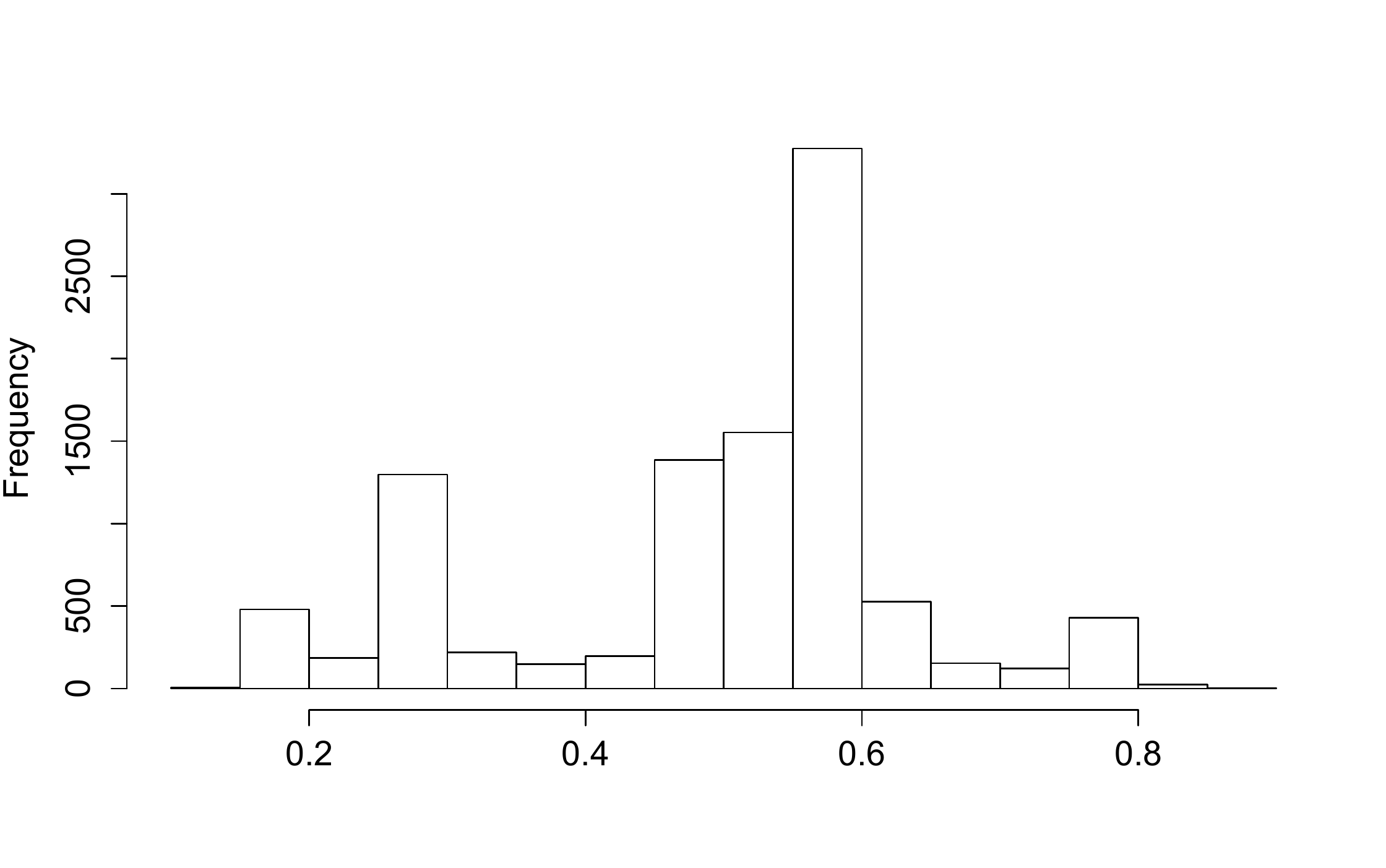}
\subcaption{CUSUM $\tilde{\theta}_C$} 
\end{subfigure}
 \begin{subfigure}[c]{0.5\textwidth}
\includegraphics[width=7cm, height=5cm]{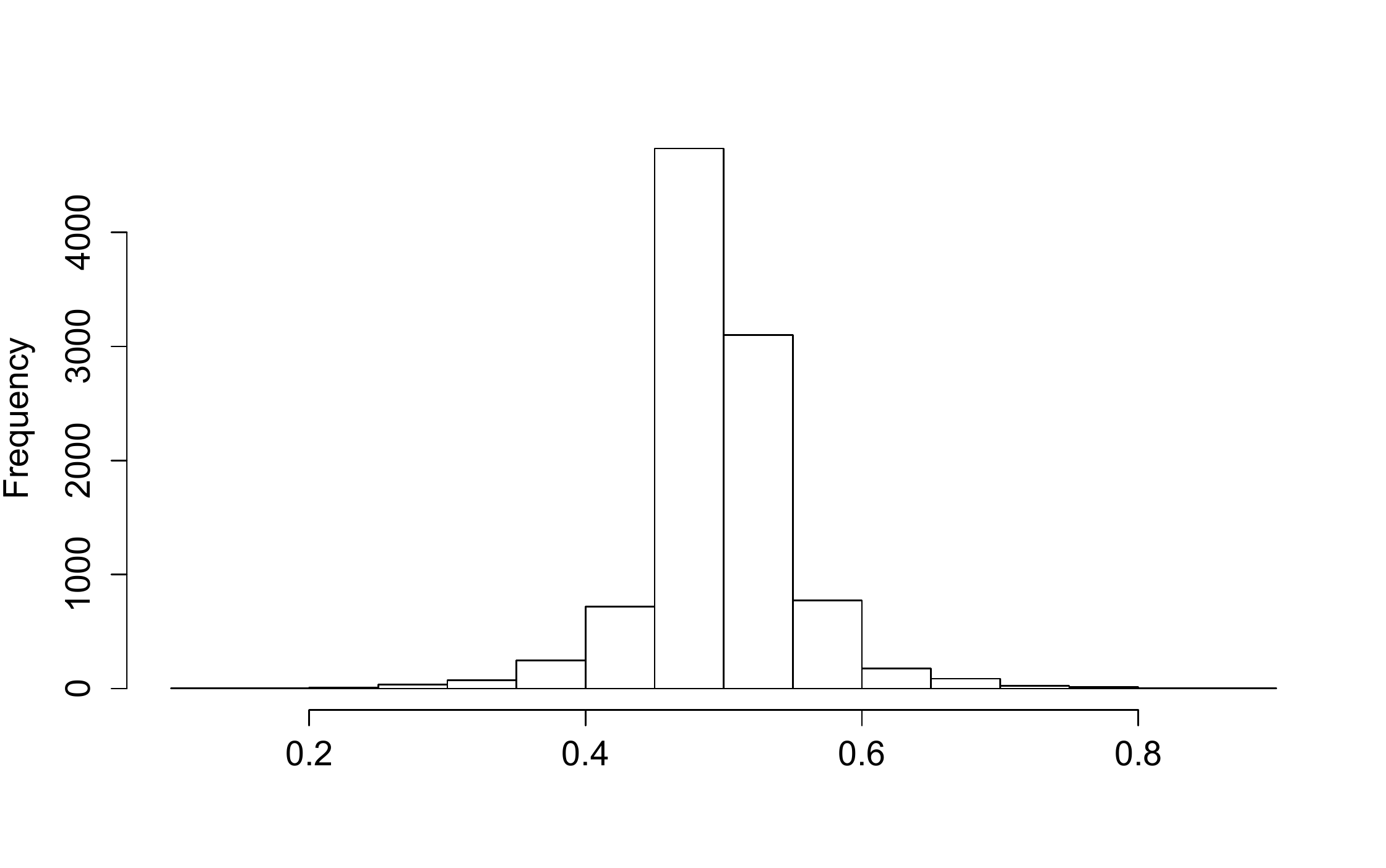}
\subcaption{Wilcoxon $\hat{\theta}$} 
\end{subfigure}
\caption{Histogram based on 10.000 values of $\tilde{\theta}_C$ and $\hat{\theta}$ for the model (\ref{model_bp_simu}) with normal innovations $\epsilon_i\sim\NIID(0,1)$, $\Delta=1$, $\theta=0.5$, $n=200$ and outliers.}\label{figure.outliers.hist}
\end{figure}

\begin{table}
\begin{tabular}{l|c|r|r|r|r|r|r|r|r}
\multicolumn{2}{c|}{} & \multicolumn{2}{|c|}{n=50} & \multicolumn{2}{|c|}{n=100} & \multicolumn{2}{|c|}{n=200}  & \multicolumn{2}{|c}{n=500}\\\hline
Innovations &  & \multicolumn{1}{|c|}{C} & \multicolumn{1}{|c|}{W}& \multicolumn{1}{|c|}{C} & \multicolumn{1}{|c}{W} & \multicolumn{1}{|c|}{C} & \multicolumn{1}{|c}{W}& \multicolumn{1}{|c|}{C} & \multicolumn{1}{|c}{W}\\\hline
normal  		& mean 	&0.50 &0.50	&0.50 &0.50	&0.50 &0.50 &0.50 &0.50\\
				& sd	&0.12 &0.12	&0.09 &0.09	&0.05 &0.05 &0.02 &0.02\\
$t_1$			& mean 	&0.52 &0.50	&0.51 &0.50	&0.51 &0.50	&0.50 &0.50\\
				& sd	&0.23 &0.20	&0.24 &0.19	&0.24 &0.17	&0.25 &0.14\\
normal with		& mean 	&0.50 &0.49	&0.50 &0.50	&0.50 &0.50	&0.51 &0.50\\
outliers		& sd	&0.17 &0.13	&0.16 &0.09	&0.15 &0.06	&0.09 &0.02\\
		
\end{tabular}
\caption{Sample mean and the sample standard deviation of $\hat{\theta}$ and $\tilde{\theta}_C$ based on 10.000 replications for the normal, normal with outliers and $t_1$-distributed innovations, $\Delta=1$ and $\theta=0.5$.}\label{table.normal.t1.outliers}
\end{table}   

\FloatBarrier 

\section{Useful properties of the Wilcoxon test statistic and proof of \thref{consistence_of_bp_estimator}}\label{Useful Properties}

This section presents some useful properties of the Wilcoxon test statistic and the proof of \thref{consistence_of_bp_estimator}.\par 
Throughout the paper without loss of generality, we assume that $\mu = 0$ and $\Delta_n>0$. We let $C$ denote a generic non-negative constant, which may vary from time to time. 
The notation $a_n \sim b_n$ means that two sequences $a_n$ and $b_n$ of real numbers have property $a_n/b_n\rightarrow c$, as $n\rightarrow \infty$, where $c\neq 0$ is a constant. $\|g\|_{\infty}=\sup_{x}|g(x)|$ stands for the supremum norm of function $g$. By $\xrightarrow{d}$ we denote the convergence in distribution, by $\rightarrow_p$ the convergence in probability and by $\overset{d}{=}$ we denote equality in distribution.

\subsection{U-statistics and Hoeffding decomposition}\label{Sub_U-statistics and Hoeffding decomposition}

The Wilcoxon test statistic $W_n(k)$ in (\ref{Wilcoxon-TS}) under the change-point model (\ref{model_bp}) can be decomposed into two terms
\begin{align}
&W_{n}(k) =\sum_{i=1}^{k}\sum_{j=k+1}^{n}( 1_{\{X_i \leq X_j\}}-1/2 )\nonumber\\
&= \begin{cases}
\sum_{i=1}^{k}\sum_{j=k+1}^{n}( 1_{\{Y_i \leq Y_j\}}-1/2 ) + \sum_{i=1}^{k}\sum_{j=k^*+1}^{n} 1_{\{Y_j < Y_i \leq Y_j+\Delta_n\}}, &1\leq k \leq k^*\\
\sum_{i=1}^{k}\sum_{j=k+1}^{n}( 1_{\{Y_i \leq Y_j\}}-1/2 ) + \sum_{i=1}^{k^*}\sum_{j=k+1}^{n} 1_{\{Y_j < Y_i \leq Y_j+\Delta_n\}}, &k^* < k \leq n,
\end{cases}\nonumber\\
&= \begin{cases}
U_n(k) + U_n(k,k^*), & 1\leq k \leq k^*\\
U_n(k) + U_n(k^*,k), & k^* < k \leq n, \label{decompose_Wn}
\end{cases}
\end{align}
where
\begin{align}
U_n(k) &= \sum_{i=1}^{k}\sum_{j=k+1}^{n}( 1_{\{Y_i \leq Y_j\}}-1/2 ), &1\leq k \leq n,\label{decompose_Un}\\
U_n(k,k^*) &= \sum_{i=1}^{k}\sum_{j=k^*+1}^{n} 1_{\{Y_j < Y_i \leq Y_j+\Delta_n\}}, &1\leq k \leq k^*,\label{decompose_Un_k1}\\
U_n(k^*,k) &= \sum_{i=1}^{k^*}\sum_{j=k+1}^{n} 1_{\{Y_j < Y_i \leq Y_j+\Delta_n\}}, &k^* < k \leq n.\label{decompose_Un_k2}
\end{align}

The first term $U_n(k)$ depends only on the underlying process $(Y_j)$, while the terms $U_n(k,k^*)$ and $U_n(k^*,k)$ depend in addition on the change-point time $k^*$ and the magnitude $\Delta_n$ of the change in the mean.\par 
The term $U_n(k)$ can be written as a second order U-statistic
\begin{equation*}
 U_n\left(k\right) = \sum_{i=1}^{k}\sum_{j=k+1}^{n}\left(h\left(Y_i,Y_j\right)-\Theta\right), \qquad 1\leq k \leq n,
\end{equation*}
with the kernel function $h\left(x,y\right) = 1_{\left\{x\leq y\right\}}$ and the constant $\Theta = \E h\left(Y_1',Y_2'\right) = 1/2$, where $Y_1'$ and $Y_2'$ are independent copies of $Y_1$.\par 
We apply to $U_n\left(k\right)$ Hoeffding's decomposition of U-statistics established by \citet{Hoeffding.1948}. It allows to write the kernel function as the sum
\begin{equation}\label{hoeffding_decomposition_h}
h\left(x,y\right) = \Theta + h_1\left(x\right) + h_2\left(y\right) + g\left(x,y\right),
\end{equation}
where
\begin{align*}
h_1\left(x\right) &= \E h\left(x,Y_2'\right) - \Theta = 1/2 - \F\left(x\right),\qquad
h_2\left(y\right)= \E h\left(Y_1',y\right) - \Theta = \F\left(y\right) - 1/2,\\
g\left(x,y\right) &= h\left(x,y\right) -   h_1\left(x\right) - h_2\left(y\right)-\Theta.
\end{align*}
By definition of $h_1$ and $h_2$, $\E h_1(Y_1)=0$ and $\E h_2(Y_1)=0$. Hence, $\E g(x,Y_1)=\E g(Y_1,y)=0$, i.e. $g(x,y)$ is a degenerate kernel.\par 

The term $U_n(k,k^*)$ in (\ref{decompose_Un_k1}) (and $U_n(k^*,k)$ in (\ref{decompose_Un_k2})) can be written as a U-statistic
\[
U_n(k,k^*) = \sum_{i=1}^{k}\sum_{j=k^*+1}^{n}h_n(Y_i,Y_j),\qquad 1\leq k \leq k^*,
\]
 with the kernel $h_n(x,y) = h(x,y+\Delta_n)-h(x,y) = 1_{\{y<x\leq y+\Delta_n\}}$. The Hoeffding decomposition allows to write the kernel as
\begin{equation}\label{hoeffding_decomposition_hn}
h_n\left(x,y\right) = \Theta_{\Delta_n} + h_{1,n}\left(x\right) + h_{2,n}\left(y\right) + g_n\left(x,y\right),
\end{equation}
with $\Theta_{\Delta_n} = \E 1_{\{Y_2'\leq Y_1' \leq Y_2'+\Delta_n\}},$
\begin{align*}
h_{1,n}\left(x\right) &= \E h_n\left(x,Y_2'\right) - \Theta_{\Delta_n} = \F\left(x\right) - \F\left(x-\Delta_n\right) - \Theta_{\Delta_n}, \\
h_{2,n}\left(y\right)&= \E h_n\left(Y_1',y\right) - \Theta_{\Delta_n}  = \F\left(y+\Delta_n\right) - \F\left(y\right) - \Theta_{\Delta_n},\\
g_n\left(x,y\right) &= h_n\left(x,y\right) -   h_{1,n}\left(x\right) - h_{2,n}\left(y\right)-\Theta_{\Delta_n}.
\end{align*}
By assumption the distribution function $\F$ of $Y_1$ has bounded probability density $f$ and bounded second derivative. This allows to specify the asymptotic behaviour of $\Theta_{\Delta_n}$, as $n\rightarrow\infty,$
\begin{align}
\Theta_{\Delta_n} &= \E 1_{\{Y_2' < Y_1' \leq Y_2'+\Delta_n\}} = \p\left(Y_2' < Y_1' \leq Y_2'+\Delta_n\right)\nonumber\\ 
&= \int_{\mathbb{R}}\left(\F\left(y+\Delta_n\right)-\F\left(y\right)\right)dF\left(y\right) = \Delta_n \bigg(\int_{\mathbb{R}}f^2\left(y\right)dy + o(1)\bigg). \label{remark_theta_delta_n}
\end{align}
Note that $\E h_{1,n}(Y_1)=0$ and $\E h_{2,n}(Y_1)=0$. Therefore, $g_n(x,y)$ is a degenerate kernel, i.e.  $\E g_n(x,Y_1)=\E g_n(Y_1,y)=0$. Furthermore, $\|h_{1,n} \|_{\infty}\rightarrow 0$, as $n\rightarrow \infty$, since 
\begin{equation}\label{abschaetzung_h1n_kleiner_delta}
|h_{1,n}(x)| \leq |\F(x)-\F(x-\Delta_n)-\Theta_{\Delta_n}|\leq C\Delta_n + \Theta_{\Delta_n} \leq C\Delta_n,
\end{equation}
where $C>0$ is a constant and $\Delta_n\rightarrow 0$, as $n\rightarrow\infty$.

\subsection{1-continuity property of kernel functions \boldmath{$h$} and \boldmath{$h_n$}}\label{Sub_1-continuity_h_hn}

Asymptotic properties of near epoch dependent processes $(Y_j)$ introduced in Section \ref{Main Results} are well investigated in the literature, see e.g. \citet{borovkova.2001}. In the context of change-point estimation we are interested in asymptotic properties of the variables $h(Y_i,Y_j)$, where $h(x,y)=1_{\{x\leq y\}}$ is the Wilcoxon kernel, and also in properties of the terms $h_1(Y_j)$ and $h_{1,n}(Y_j)$ of the Hoeffding decomposition of the kernels in (\ref{hoeffding_decomposition_h}) and (\ref{hoeffding_decomposition_hn}). We will need to show that the variables $(h(Y_i,Y_j))$, $(h_1(Y_j))$ and $(h_{1,n}(Y_j))$ retain some properties of $(Y_j)$. To derive them, we will use the fact that the kernels $h$ in (\ref{hoeffding_decomposition_h}) and $h_n$ in (\ref{hoeffding_decomposition_hn}) satisfy the $1$-continuity condition introduced by \citet{borovkova.2001}.

\begin{definition}\thlabel{1-continuous}\par
We say that the kernel $h\left(x,y\right)$ is $1$-continuous with respect to a distribution of a stationary process $(Y_j)$ if there exists a function $\phi(\epsilon)\geq 0$, $\epsilon\geq 0$ such that $\phi\left(\epsilon\right)\rightarrow 0$, $\epsilon\rightarrow 0$, and for all $\epsilon > 0$ and $k\geq 1$
\begin{align}
\E \left(\left| h\left(Y_1,Y_k\right)-h\left(Y_1',Y_k\right) \right| 1_{\left\{\left|Y_1-Y_1'\right|\leq \epsilon\right\}}\right) & \leq \phi\left(\epsilon\right),\label{1-cont_11}\\
\E \left(\left| h\left(Y_k,Y_1\right)-h\left(Y_k,Y_1'\right) \right| 1_{\left\{\left|Y_1-Y_1'\right|\leq \epsilon\right\}}\right) & \leq \phi\left(\epsilon\right),\nonumber
\end{align}
and
\begin{align}
\E \left(\left| h\left(Y_1,Y_2'\right)-h\left(Y_1',Y_2'\right) \right| 1_{\left\{\left|Y_1-Y_1'\right|\leq \epsilon\right\}}\right) & \leq \phi\left(\epsilon\right),\label{1-cont_22}\\
\E \left(\left| h\left(Y_2',Y_1\right)-h\left(Y_2',Y_1'\right) \right| 1_{\left\{\left|Y_1-Y_1'\right|\leq \epsilon\right\}}\right) & \leq \phi\left(\epsilon\right),\nonumber
\end{align}
where $Y_2'$ is an independent copy of $Y_1$ and $Y_1'$ is any random variable that has the same distribution as $Y_1$.
\end{definition}

For a univariate function $g(x)$ we define the $1$-continuity property as follows.

\begin{definition}\par
The function $g\left(x\right)$ is $1$-continuous with respect to a distribution of a stationary process $(Y_j)$ if there exists a function $\phi(\epsilon)\geq 0$, $\epsilon\geq 0$ such that $\phi\left(\epsilon\right)\rightarrow 0$, $\epsilon\rightarrow 0$, and for all $\epsilon > 0$
\begin{align}\label{1-cont_3}
\E \left(\left| g\left(Y_1\right)-g\left(Y_1'\right) \right| 1_{\left\{\left|Y_1-Y_1'\right|\leq \epsilon\right\}}\right) & \leq \phi\left(\epsilon\right),
\end{align}
where $Y_1'$ is any random variable that has the same distribution as $Y_1$.
\end{definition}

\thref{lemma_wilcoxon_1_continuous} below establishes the $1$-continuity of functions $h(x,y)=1_{\{x\leq y\}}$ and $h_n(x,y)=1_{\{y<x\leq y+\Delta_n\}}$, $n\geq 1$. For $h_n$, $n\geq 1$ we assume that (\ref{1-cont_11}) and (\ref{1-cont_22}) hold with the same $\phi(\epsilon)$ for all $n\geq 1$. We start the proof by showing the $1$-continuity of the more general kernel function 
$h(x,y;t)=1_{\{x-y\leq t\}}$.

\begin{lemma}\thlabel{lemma_kernel:t_1_continuous}
Let $(Y_j)$ be a stationary process, $Y_1$ have distribution function $F$ which has bounded first and second derivative and $Y_1-Y_k$, $k\geq 1$ satisfy (\ref{Bed_gem_Verteilung}). Then the function $h(x,y;t)=1_{\{x-y\leq t\}}$ is $1$-continuous with respect to the distribution function of $(Y_j)$. 
\end{lemma}

\begin{proof}
The proof is similar to the proof of $1$-continuity of the kernel function $h(x,y;t)=1_{\{|x-y|\leq t\}}$ given in Example 2.2 of \citet{borovkova.2001}.\par 
 Note that $1_{\{Y_1- Y_k\leq t\}}-1_{\{Y_1'- Y_k\leq t\}}=0$ if $Y_1-Y_k\leq t$ and $Y_1'-Y_k\leq t$; or $Y_1-Y_k> t$ and $Y_1'-Y_k> t$. The difference is not zero if  $Y_1-Y_k\leq t$ and $Y_1'-Y_k> t$; or $Y_1-Y_k> t$ and $Y_1'-Y_k\leq t$.
Let $|Y_1-Y_1'|\leq \epsilon$, where $\epsilon>0$. Then $Y_1-Y_k<t-\epsilon$ implies $Y_1'-Y_k<t$, and $Y_1-Y_k>t+\epsilon$ implies $Y_1'-Y_k>t$.\par 
Hence, $|1_{\{Y_1- Y_k\leq t\}}-1_{\{Y_1'- Y_k\leq t\}}|1_{\left\{\left|Y_1-Y_1'\right|\leq \epsilon\right\}}\leq 1_{\left\{t-\epsilon\leq Y_1-Y_k\leq t+\epsilon\right\}}$. Therefore,
\begin{equation}\label{ineq_1_continuous}
\E\left(\left|1_{\{Y_1- Y_k\leq t\}}-1_{\{Y_1'- Y_k\leq t\}}\right|1_{\left\{\left|Y_1-Y_1'\right|\leq \epsilon\right\}}\right) \leq \p\left(t-\epsilon\leq Y_1-Y_k\leq t+\epsilon\right)\leq C_1\epsilon,
\end{equation}
because of assumption (\ref{Bed_gem_Verteilung}).
Similar argument yields
\begin{align*}
\E\left(\left|1_{\{Y_k- Y_1\leq t\}}-1_{\{Y_k- Y_1'\leq t\}}\right|1_{\left\{\left|Y_1-Y_1'\right|\leq \epsilon\right\}}\right)&\leq \p\left(t-\epsilon\leq Y_1-Y_k\leq t+\epsilon\right)\leq C_1\epsilon, \\
\E\left(\left|1_{\{Y_1- Y_2'\leq t\}}-1_{\{Y_1'- Y_2'\leq t\}}\right|1_{\left\{\left|Y_1-Y_1'\right|\leq \epsilon\right\}}\right) &\leq \p(t-\epsilon \leq Y_1-Y_2'\leq t+\epsilon)\leq C_2\epsilon,\\
\E\left(\left|1_{\{Y_2'- Y_1\leq t\}}-1_{\{Y_2'- Y_1'\leq t\}}\right|1_{\left\{\left|Y_1-Y_1'\right|\leq \epsilon\right\}}\right) &\leq \p(t-\epsilon \leq Y_1-Y_2'\leq t+\epsilon)\leq C_2\epsilon,
\end{align*}
where $Y_2'$ is an independent copy of $Y_1$, noting that by the mean value theorem and $|d\F(y)/dy|\leq C$,
\begin{align*}
\p(t-\epsilon\leq Y_1-Y_2'\leq t+\epsilon) &= \int_{\mathbb{R}}\left(\F\left(y+t+\epsilon\right)-\F\left(y+t-\epsilon\right)\right)dF\left(y\right)\\
&\leq C\epsilon\int_{\mathbb{R}}f(y)dy = C_2\epsilon.
\end{align*}
These bounds imply (\ref{1-cont_11}) and (\ref{1-cont_22}) with $\phi(\epsilon)=C\epsilon$, where $C$ does not depend on $t$. This completes the proof.
\end{proof}

\begin{coro}\thlabel{lemma_wilcoxon_1_continuous}
Assume that assumptions of \thref{lemma_kernel:t_1_continuous} are satisfied. Then, \begin{enumerate}[(i)]
\item Function $h(x,y)=1_{\{x\leq y\}}$ is $1$-continuous with respect to the distribution function of $(Y_j)$. 
\item Function $h_n(x,y)=1_{\{y<x\leq y+\Delta_n\}}$ is $1$-continuous with respect to the distribution function of $(Y_j)$. 
\end{enumerate}
\end{coro} 

\begin{proof}
\textit{(i)} follows from \thref{lemma_kernel:t_1_continuous}, noting that $1_{\{x\leq y\}}=h(x,y;0)$.\par 
\textit{(ii)} We need to verify (\ref{1-cont_11}) and (\ref{1-cont_22}). Write $h_n(x,y)=h(x,y)-h(x,y+\Delta_n)=1_{\{x\leq y\}}-1_{\{x\leq y+\Delta_n\}}$. Then by (\ref{ineq_1_continuous}),
\begin{multline*}
\E\big(|h_n(Y_1,Y_k)-h_n(Y_1',Y_k)|1_{\left\{\left|Y_1-Y_1'\right|\leq \epsilon\right\}}\big)
\leq \E\big(|1_{\{Y_1\leq Y_k\}}-1_{\{Y_1'\leq Y_k\}}|1_{\left\{\left|Y_1-Y_1'\right|\leq \epsilon\right\}}\big)
\\ + \E\big(|1_{\{Y_1\leq Y_k+\Delta_n\}}-1_{\{Y_1'\leq Y_k+\Delta_n\}}|1_{\left\{\left|Y_1-Y_1'\right|\leq \epsilon\right\}}\big)
\leq C\epsilon.
\end{multline*}
Similar argument yields $\E\big(|h_n(Y_k,Y_1)-h_n(Y_k,Y_1')|1_{\left\{\left|Y_1-Y_1'\right|\leq \epsilon\right\}}\big) \leq C\epsilon,$
\begin{align*}
\E\big(|h_n(Y_1,Y_2')-h_n(Y_1',Y_2')|1_{\left\{\left|Y_1-Y_1'\right|\leq \epsilon\right\}}\big) &\leq C\epsilon,\\
\E\big(|h_n(Y_2',Y_1)-h_n(Y_2',Y_1')|1_{\left\{\left|Y_1-Y_1'\right|\leq \epsilon\right\}}\big)&\leq C\epsilon.
\end{align*}
Hence, (\ref{1-cont_11}) and (\ref{1-cont_22}) hold with $\phi(\epsilon)=C\epsilon$.
\end{proof}

Note that condition (\ref{Bed_gem_Verteilung}) is satisfied if variables $(Y_1,Y_k)$, $k\geq 1$, have joint probability densities that are bounded by the same constant $C$ for all $k$. If the joint density does not exist, for examples of verification of condition (\ref{Bed_gem_Verteilung}) see pages 4315, 4316 of \citet{borovkova.2001}.\par\bigskip

Lemma 2.15 of \citet{borovkova.2001} yields that if a general function $h(x,y)$ is $1$-continuous, i.e. satisfies (\ref{1-cont_11}) and (\ref{1-cont_22}) with function $\phi(\epsilon)$ then $\E h\left(x,Y_2'\right)$, where $Y_2'$ is an independent copy of $Y_1$, is also $1$-continuous and satisfies the condition in (\ref{1-cont_3}) with the same function $\phi(\epsilon)$. Hence, $h_i(x)$ and $h_{i,n}(x)$, $i=1,2$ are $1$-continuous and satisfy the condition in (\ref{1-cont_3}) with $\phi(\epsilon)=C\epsilon$.\par\bigskip

Next we turn to $1$-continuity property of $g(x,y)$. 
By Hoeffding decomposition (\ref{hoeffding_decomposition_h}), $g(x,y)=h(x,y)-\Theta-h_1(x)-h_2(y)$. Since $h(x,y)$, $h_1(x)$ and $h_2(x)$ in (\ref{hoeffding_decomposition_h}) are $1$-continuous and satisfy (\ref{1-cont_11}), (\ref{1-cont_22}) and (\ref{1-cont_3}) with the same function $\phi(\epsilon)=C\epsilon$, then $g(x,y)$ is also $1$-continuous with function $\phi(\epsilon)=C\epsilon$. Indeed,
\begin{multline*}
\E \left(|g(Y_1,Y_k)-g(Y_1',Y_k)|1_{\left\{\left|Y_1-Y_1'\right|\leq \epsilon\right\}} \right)\\
\leq \E \left(|h(Y_1,Y_k)-h(Y_1',Y_k)|1_{\left\{\left|Y_1-Y_1'\right|\leq \epsilon\right\}} \right) + \E \left(|h_1(Y_1)-h_1(Y_1')|1_{\left\{\left|Y_1-Y_1'\right|\leq \epsilon\right\}} \right)
\leq 2\phi(\epsilon)
\end{multline*}
and similarly, $\E \big(|g(Y_k,Y_1)-g(Y_k,Y_1')|1_{\left\{\left|Y_1-Y_1'\right|\leq \epsilon\right\}} \big)\leq 2\phi(\epsilon)$.\par 
Using the same argument, it follows that the function $g_n(x,y)=h_n(x,y)-\Theta_{\Delta_n}-h_{1,n}(x)-h_{2,n}(x)$ in the Hoeffding decomposition (\ref{hoeffding_decomposition_hn}) is also $1$-continuous and satisfies (\ref{1-cont_11}), (\ref{1-cont_22}) with $\phi(\epsilon)=C\epsilon$.\par

\subsection{NED property of \boldmath{$(h_1(Y_j))$} and \boldmath{$(h_{1,n}(Y_j))$}} \label{Properties of the kernel}

In Proposition 2.11 of \citet{borovkova.2001} it is shown that if $(Y_j)$ is $L_1$ NED on a stationary absolutely regular process $(Z_j)$ with approximation constants $a_k$ and $g(x)$ is $1$-continuous with function $\phi$, then $(g(Y_j))$ is also $L_1$ NED on $(Z_j)$ with approximation constants $\phi\left(\sqrt{2a_k}\right)+2\sqrt{2a_k}||g||_{\infty}$.\par
Thus, the processes $(h_1(Y_j))$ and $(h_2(Y_j))$ in (\ref{hoeffding_decomposition_h}) and $(h_{1,n}(Y_j))$ and $(h_{2,n}(Y_j))$ in (\ref{hoeffding_decomposition_hn}) are $L_1$ NED processes with approximation constants $a'_k=C\sqrt{a_k}$.\par \bigskip

Corollary 3.2 of \citet{wooldridge.1988} provides a functional central limit theorem for partial sum process $\sum_{i=1}^k \tilde{Y}_i$, $k\geq 1$, where $(\tilde{Y}_j)$ is $L_2$ NED on a strongly mixing process $(\tilde{Z}_j)$. To apply this result to $(h_1(Y_j))$ which is $L_1$ NED on $(Z_j)$ with approximation constants $a_k'$, we need to show that $(h_1(Y_j))$ is also $L_2$ NED process. Note that the variables $\eta_k := h_1(Y_1)-\E(h_1(Y_1)|\mathcal{G}_{-k}^k)$ have property
\begin{multline*}
\E\eta_k^2 = \E\big(\eta_k^2 1_{\big\{|\eta_k|\leq {a_k'}^{-\frac{1}{2}}\big\}}\big)+ \E\big(\eta_k^2 1_{\big\{|\eta_k|> {a_k'}^{-\frac{1}{2}}\big\}}\big)\\
\leq {a_k'}^{-\frac{1}{2}}\E|\eta_k| +  \sqrt{a_k'}\E|\eta_k|^{4}\leq \sqrt{a_k'} + \sqrt{a_k'} C =: a_k''.
\end{multline*}
The last inequality holds, because by \thref{NED} of $L_1$ near epoch dependence, $\E|h_1(Y_1)-\E(h_1(Y_1)|\mathcal{G}_{-k}^k)|\leq a_k'$ and because $|h_1(Y_1)|\leq 1/2$. Therefore the process $(h_1(Y_j))$ is $L_2$ NED on $(Z_j)$ with approximation constant $a_k''$. Since absolute regular process $(Z_j)$ is strongly mixing process, from Corollary 3.2 of \citet{wooldridge.1988}, we obtain
\begin{equation*}
\bigg( \frac{1}{n^{1/2}}\sum_{i=1}^{\left[nt\right]}h_1\left(Y_i\right) \bigg)_{0\leq t \leq 1} \xrightarrow{d} \left( \sigma W\left(t\right) \right)_{0\leq t \leq 1},
\end{equation*}
where $W\left(t\right)$ is a Brownian motion and $\sigma^2 =\sum_{k=-\infty}^{\infty}\Cov(h_1(Y_0),h_1(Y_k))$.\par
Since $h_2(x) = -h_1(x)$, all properties of $(h_1(Y_j))$ remain valid also for $(h_2(Y_j))$.

\subsection{Proof of \thref{consistence_of_bp_estimator}}\label{Proof}

First we show consistency property $|k^*-\hat{k}|=o_P(k^*)$ of the estimate $\hat{k}=\argmax_{1\leq k\leq n}|W_n(k)|$. To prove it, we verify that for any $\epsilon>0$,
\begin{equation}\label{zz_consistency_1}
\lim_{n\rightarrow \infty}\p \left( |k^*-\hat{k}| \leq \epsilon k^*\right) = 1.
\end{equation}
This means that the estimated value $\hat{k}$ with probability tending to $1$ is in a neighbourhood of the true value $k^*$:
\[
\p\big(\hat{k}\in[k^*(1-\epsilon),k^*(1+\epsilon)]\big) \rightarrow 1.
\]
We will show that as $n\rightarrow\infty$, 
\begin{equation}\label{zz_consistency_2}
\p\Big( \max_{k:|k-k^*|\geq \epsilon k^*}|W_n(k)| < |W_n(k^*)| \Big) \rightarrow 1.
\end{equation}
Since $|W_n(k^*)|\leq \max_{k:|k^*-k|\leq \epsilon k^*}|W_n(k)|$, this proves (\ref{zz_consistency_1}).\par
\bigskip
By (\ref{decompose_Wn}),
\begin{equation*}
W_{n}(k)= \begin{cases}
U_n(k) + U_n(k,k^*),\qquad 1\leq k \leq k^*\\
U_n(k) + U_n(k^*,k),\qquad k^* < k \leq n.
\end{cases}
\end{equation*}

\thref{Theorem_DFGW} implies $\max_{1\leq k \leq n} \left|U_n(k)\right| = O_P\left(n^{3/2}\right)$ and \thref{Lemma_abschaetzung_bp} below yields
\begin{align*}
\max_{1\leq k \leq k^*}\left| U_n(k,k^*)- k\left(n-k^*\right)\Theta_{\Delta_n}\right| &=  o_P\left(n^{3/2}\right),\\
\max_{k^*\leq k \leq n}\left|U_n(k^*,k)-k^*\left(n-k\right)\Theta_{\Delta_n}\right| &= o_P\left(n^{3/2}\right).
\end{align*}
Hence,
\begin{align*}
W_n(k^*) &=k^*(n-k^*)\Theta_{\Delta_n} + (U_n(k^*,k^*)-k^*(n-k^*)\Theta_{\Delta_n})+U_n(k^*)\\ &= k^*(n-k^*)\Theta_{\Delta_n} + O_P(n^{3/2}),\\
\max_{1\leq k \leq k^*(1-\epsilon)} |W_n(k)| &\leq (1-\epsilon)k^*(n-k^*)\Theta_{\Delta_n} + O_P(n^{3/2}),\\
\max_{(1+\epsilon)k^*\leq k \leq n} |W_n(k)| &\leq k^*(n-(1+\epsilon)k^*)\Theta_{\Delta_n} + O_P(n^{3/2}).
\end{align*}
Thus,
\begin{equation*}
|W_n(k^*)|-\max_{k:|k^*-k|\geq \epsilon k^*} |W_n(k)| \geq \epsilon \delta_n + O_P(n^{3/2}),
\end{equation*}
where $\delta_n = k^*\min(n-k^*,{k^*})\Theta_{\Delta_n}$.\par 

By definition $k^*=[n\theta]\sim n\theta$, and by (\ref{remark_theta_delta_n}) and (\ref{assumption_size_change_infty}), $\sqrt{n}\Theta_{\Delta_n}\sim c\sqrt{n}\Delta_n\rightarrow\infty$. Hence, $\delta_n^{-1}=o(n^{-3/2})$ and $\epsilon \delta_n + O_P(n^{3/2}) = \epsilon\delta_n(1+O_P(n^{3/2}\delta_n^{-1}))=\epsilon\delta_n(1+o_P(1))$ which proves (\ref{zz_consistency_2}).\par 
 \bigskip
 Next we establish the rate of convergence in (\ref{rate_consistency_k_hat}), $k^*-k=O_P(1/\Delta_n^2)$. Set $a(n)=\frac{M}{\Delta_n^2}$. Then for fixed $M>0$, $a(n)\rightarrow\infty$, as $n\rightarrow\infty$. We will verify that 
\begin{equation*}
\lim_{n\rightarrow \infty}\p \left( |k^*-\hat{k}| \leq a(n)\right) \rightarrow 1, \qquad \text{as} \; M\rightarrow\infty, 
\end{equation*}
which implies (\ref{rate_consistency_k_hat}). As in (\ref{zz_consistency_2}), we prove this by showing
 \begin{equation}\label{zz_consistency_3}
\lim_{n\rightarrow \infty}\p\Big( \max_{k:|k-k^*|\geq a(n)}|W_n(k)| < |W_n(k^*)| \Big) \rightarrow 1, \qquad \text{as} \; M\rightarrow\infty.
\end{equation}
 
Define $V_k := W_{n}^2\left(k\right)-W_{n}^2\left(k^*\right)$. If $|W_{n}\left(k\right)|$ attains its maximum at $k'$, it is easy to see that $V_k$ attains its maximum at the same $k'$. Hence, $ \hat{k}= \min \{ k: \max_{1\leq l \leq n} |W_n(l)|=|W_n(k)| \} = \min \{ k: \max_{1\leq l \leq n} V_l  = V_k\}$. Thus, instead of (\ref{zz_consistency_3}) it remains to show that
 \begin{equation}\label{zz_consistency_4}
\lim_{n\rightarrow \infty}\p\Big( \max_{k:|k-k^*|\geq a(n)}V_k < 0 \Big) \rightarrow 1, \qquad M\rightarrow\infty.
\end{equation}
 
Define $\tilde{k} := \min \{ k: |k-k^*|\leq \epsilon k^*; V_k = \max_{n\alpha\leq l \leq n\beta} V_l\}$. Since by (\ref{zz_consistency_1}) $\hat{k}$ is a consistent estimator of $k^*$, it holds $\lim_{n\rightarrow\infty}\p(\hat{k}=\tilde{k})=1$.\par 
So, in the proof of (\ref{zz_consistency_4}) it suffices to consider $\max$ over $k$, such that $|k-k^*|\leq \epsilon k^*$, $|k-k^*|\geq a(n)$, which corresponds to $(1-\epsilon)k^*\leq k \leq k^*-a(n)$ and $k^*+a(n)<k\leq (1+\epsilon)k^*$.\par \bigskip

Let us start with $(1-\epsilon)k^*\leq k \leq k^*-a(n)$. Since $k^*-k>0$, relation (\ref{zz_consistency_4}) holds for such $k$, if
 \begin{equation}\label{zz_consistency_5}
\lim_{n\rightarrow \infty}\p\Big( \max_{(1-\epsilon)k^*\leq k \leq k^*-a(n)}\frac{V_k}{(n(k^*-k))^2} < 0 \Big) \rightarrow 1, \qquad M\rightarrow\infty.
\end{equation}

Note that
\begin{align}\label{zerlegung_v_k}
\frac{-V_k}{(n(k^*-k))^2} &= \frac{W_n^2(k^*)-W_n^2(k)}{(n(k^*-k))^2}\nonumber\\
&= -\bigg(\frac{W_n(k^*)-W_n(k)}{n(k^*-k)}\bigg)^2+2\frac{W_n(k^*)-W_n(k)}{n(k^*-k)}\frac{W_n(k^*)}{n(k^*-k)}.
\end{align}

By (\ref{decompose_Wn}), $W_n(k)=U_n(k)+U_n(k,k^*)$. Then,
\begin{equation*}
\frac{W_n(k^*)-W_n(k)}{n(k^*-k)} = \frac{n-k^*}{n}\Theta_{\Delta_n} + \delta_{1,k} + \delta_{2,k},
\end{equation*}
where
\[
\delta_{1,k}= \frac{U_n(k^*)-U_n(k)}{n(k^*-k)}, \qquad \delta_{2,k}=\frac{U_n(k^*,k^*)-U_n(k,k^*)}{n(k^*-k)}-\frac{n-k^*}{n}\Theta_{\Delta_n}.
\]
Observe that by (\ref{remark_theta_delta_n}), $\Theta_{\Delta_n}\sim c_{*}\Delta_n$, $c_{*}>0$, and $k^*/n\rightarrow\theta$. Therefore, $(n-k^*)/n\Theta_{\Delta_n}\sim c_0\Delta_n$, where $c_0=(1-\theta)c_{*}$. Moreover, $\max_{1\leq k \leq k^*-a(n)}|\delta_{i,k}|=o_P(\Delta_n)$, $i=1,2$, by (\ref{zz_gewichtetes_maximum}) and (\ref{zz_gewichtetes_maximum_unkk}) of \thref{Lemma_abschaetzung_gewichtetes_maximum}. Hence,
\begin{equation}\label{abschaetzung_1}
\frac{W_n(k^*)-W_n(k)}{n(k^*-k)}=c_0\Delta_n(1+o_P(1)), \qquad \bigg(\frac{W_n(k^*)-W_n(k)}{n(k^*-k)}\bigg)^2=c_0^2\Delta_n^2(1+o_P(1)).
\end{equation}
In turn,
\[
\frac{W_n(k^*)}{n(k^*-k)}=\frac{U_n(k^*,k^*)}{n(k^*-k)}+\frac{U_n(k^*)}{n(k^*-k)}
\]
and 
\[
\frac{U_n(k^*,k^*)}{n} = \frac{k^*(n-k^*)\Theta_{\Delta_n}}{n}+ \frac{\delta_{3,k}}{n},
\]
where $\delta_{3,k}=U_n(k^*,k^*)-k^*(n-k^*)\Theta_{\Delta_n}$. By \thref{Lemma_abschaetzung_bp}, $\max_{1\leq k\leq k^*}\delta_{3,k}/n=o_P(n^{1/2})$. Since 
\begin{equation*}
\frac{k^*(n-k^*)\Theta_{\Delta_n}}{n}\sim k^*c_0\Delta_n \sim \theta c_0 n\Delta_n
\end{equation*}
and $\sqrt{n}=o(n\Delta_n)$, this implies
\[
\frac{U_n(k^*,k^*)}{n} = k^*c_0\Delta_n(1+o_P(1)).
\]
Next, by \thref{Theorem_DFGW} below, $U_n(k^*)=O_P(n^{3/2})$, and hence, $U_n(k^*)/n=O_P(n^{1/2})$. Therefore, $W_n(k^*)/n= k^*c_0\Delta_n(1+o_P(1))$. Hence, for $(1-\epsilon)k^*\leq k\leq k^*-a(n)$,
\begin{equation}\label{abschaetzung_2}
\frac{W_n(k^*)}{n(k^*-k)}=\frac{k^*c_0\Delta_n(1+o_P(1))}{k^*-k} \geq \frac{1}{\epsilon}c_0\Delta_n(1+o_P(1)).
\end{equation}
Using (\ref{abschaetzung_1}) and (\ref{abschaetzung_2}) in (\ref{zerlegung_v_k}), it follows
\begin{align*}
-\frac{V_k}{(n(k^*-k))^2} \geq \frac{2}{\epsilon}c_0^2\Delta_n^2(1+o_P(1))-c_0^2\Delta_n^2(1+o_P(1)) \geq \Big(\frac{2}{\epsilon}-1\Big)(c_0\Delta_n)^2(1+o_P(1)) >0.
\end{align*}
This proves (\ref{zz_consistency_5}). Similar argument yields
\begin{equation*}
\lim_{n\rightarrow\infty}\p\Big( \max_{k^*+a(n) \leq k \leq k^*(1+\epsilon)} V_k < 0 \Big) \rightarrow 1, \qquad  M\rightarrow\infty,
\end{equation*}
which completes the proof of (\ref{zz_consistency_4}) and the theorem.
$\hfill \Box $

\section{Auxiliary results}\label{Auxiliary Results}

This section contains auxiliary results used in the proof of \thref{consistence_of_bp_estimator}.\par
We establish asymptotic properties of the quantities $U_n(k)$, $U_n(k,k^*)$ and $U_n(k^*,k)$ defined in (\ref{decompose_Un})-(\ref{decompose_Un_k2}) and appearing in the decomposition (\ref{decompose_Wn}) of $W_n(k)$.\par \bigskip

The following lemma derives a H\'{a}jek-R\'{e}nyi type inequality for $L_1$ NED random variables.

\begin{lemma}\thlabel{hajek_renyi_inequality}
Let $\left(Y_j\right)$ be a stationary $L_1$ near epoch dependent process on some absolutely regular process $\left(Z_j\right)$, satisfying (\ref{condition_appr.const_regu.coeff}). Assume that $\E Y_j=0$ and $|Y_j|\leq K\leq \infty$ a.s. for some $K\geq 0$. Then, for all fixed $\epsilon>0$, for all $1\leq m \leq n$,
\begin{equation}\label{hajek_renyi}
\p \bigg(\max_{m\leq k \leq n} \frac{1}{k}\bigg|\sum_{i=1}^{k}Y_i\bigg| >\epsilon\bigg) \leq \frac{1}{\epsilon^2}\frac{C}{\sqrt{m}},
\end{equation}
where $C>0$ does not depend on $m$, $n$, $\epsilon$.
\end{lemma}

\begin{proof}
To prove (\ref{hajek_renyi}), we use the H\'{a}jek-R\'{e}nyi type inequality of \thref{theorem_KL} established in \citet{Kokoszka.2000},
\begin{multline}\label{inequality_KL}
\epsilon^2 \p\bigg(\max_{m\leq k \leq n} \frac{1}{k}\bigg|\sum_{i=1}^{k}Y_i\bigg| >\epsilon\bigg) \leq \frac{1}{m^2}\E\bigg(\sum_{i=1}^{m}Y_i\bigg)^2 + \sum_{k=m}^{n}\Big|\frac{1}{(k+1)^2}-\frac{1}{k^2}\Big|\E\bigg(\sum_{i=1}^{k}Y_i\bigg)^2\\
+ 2\sum_{k=m}^{n}\frac{1}{(k+1)^2}\E\bigg(\left|Y_{k+1}\right|\bigg|\sum_{j=1}^{k}Y_j\bigg|\bigg) + \sum_{k=m}^{n}\frac{1}{(k+1)^2}\E Y_{k+1}^2.
\end{multline}
First we bound $\E\big(\sum_{i=1}^{k}Y_i\big)^2$. Under assumptions of this lemma, by \thref{lemma_2.18} below, for $i,j\geq 0$
\begin{equation}\label{abschaetzung_covariance}
\left| \Cov \left(Y_i,Y_{i+j}\right)\right|=\left| \E \left(Y_iY_{i+j}\right)\right| \leq 4Ka_{\lfloor \frac{j}{3}\rfloor} + 2K^2\beta_{\lfloor \frac{j}{3}\rfloor} \leq C(a_{\lfloor \frac{j}{3}\rfloor}+\beta_{\lfloor \frac{j}{3}\rfloor}).
\end{equation}
By stationarity of $(Y_j)$,
\[
\left| \E \left(Y_iY_{j}\right)\right| = \left| \Cov \left(Y_i,Y_{j}\right)\right|= \left| \Cov \left(Y_0,Y_{|i-j|}\right)\right|.
\]
Hence, 
\begin{multline*}
\E\Big(\sum_{i=1}^{k}Y_i\Big)^2 = \sum_{i,j=1}^{k}\E\left(Y_iY_j\right)\leq \sum_{i,j=1}^{k}\left| \Cov \left(Y_0,Y_{|i-j|}\right)\right|\\
\leq C\sum_{i,j=1}^{k}(a_{\lfloor \frac{|i-j|}{3}\rfloor}+\beta_{\lfloor \frac{|i-j|}{3}\rfloor})\leq C\sum_{i=1}^{k}\sum_{k=0}^{\infty}(a_{\lfloor \frac{k}{3}\rfloor}+\beta_{\lfloor \frac{k}{3}\rfloor})\leq Ck,
\end{multline*}
by (\ref{abschaetzung_covariance}) and (\ref{condition_appr.const_regu.coeff}). 
Since $|Y_j|\leq K$, then
\begin{equation*}
\E\Big(\left|Y_{k+1}\right|\Big|\sum_{j=1}^{k}Y_j\Big|\Big) \leq K\E\Big(\Big|\sum_{j=1}^{k}Y_j\Big|\Big)\leq K\Big( \E\Big(\sum_{i=1}^{k}Y_i\Big)^2 \Big)^{1/2}\leq C\sqrt{k}.
\end{equation*}
Using these bounds in (\ref{inequality_KL}) together with
\[
\frac{1}{(k+1)^2}\leq \frac{1}{k^2}, \qquad \Big|\frac{1}{(k+1)^2}-\frac{1}{k^2}\Big| \leq \frac{1+2k}{(k+1)^2k^2}\leq \frac{4}{k^3},
\]
we obtain (\ref{hajek_renyi}):
\begin{equation*}
\epsilon^2 \p\bigg(\max_{m\leq k \leq n} \frac{1}{k}\bigg|\sum_{i=1}^{k}Y_i\bigg| >\epsilon\bigg) \leq C\bigg[ \frac{1}{m}+\sum_{k=m}^{n}\frac{1}{k^2} + \sum_{k=m}^{n}\frac{1}{k^{3/2}}\bigg] \leq \frac{C}{\sqrt{m}}.
\end{equation*}

\end{proof}

The next lemma establishes asymptotic bounds of the sums
\begin{equation}\label{s_k}
S_k^{(1)} = \sum_{i=1}^{k}h_{1,n}\left(Y_i\right), \qquad  S_k^{(2)} =\sum_{j=1}^{k}h_{2,n}\left(Y_j\right).
\end{equation}

\begin{lemma}\thlabel{lemma_abschaetzung_h}
Assume that $\left(Y_j\right)$ is a stationary zero mean $L_1$ near epoch dependent process on some absolutely regular process $\left(Z_j\right)$ and (\ref{condition_appr.const_regu.coeff}) holds. Furthermore, let Assumption \ref{assumption_change} be satisfied and $S_k^{(i)}$, $i=1,2$, be as in (\ref{s_k}). Then
\begin{equation}\label{abschaetzung_h1}
\max_{1\leq k \leq n}n^{-1/2}\Big| S_k^{(i)}\Big| = o_P\left(1\right), \qquad i=1,2.
\end{equation}

\end{lemma}
 
\begin{proof}
To show (\ref{abschaetzung_h1}) for $i=1$, we will use the inequality given in \thref{Billingsley}. Define
$S_k = \sum_{i=1}^{k}n^{-1/2} h_{1,n}\left(Y_i\right)$, $k\geq 1$, and set $S_0=0$. We need to evaluate $\E( S_l-S_k )^4 $ for $1\leq k < l \leq n$. Note that
\begin{equation*}
\E \left( S_l-S_k \right)^4 
 = n^{-2}\E \bigg|\sum_{i=k+1}^{l} h_{1,n}\left(Y_i\right)\bigg|^4 = n^{-2}\E \bigg|\sum_{i=1}^{l-k} h_{1,n}\left(Y_i\right)\bigg|^4,
\end{equation*}
where the last equality holds because $(h_{1,n}(Y_j))$ is a stationary process. Since $(h_{1,n}(Y_j))$ is $L_1$ NED on an absolutely regular process, see Section \ref{Properties of the kernel}, $\E h_{1,n}(Y_0)=0$ and $|h_{1,n}(x)|\leq C\Delta_n$ by (\ref{abschaetzung_h1n_kleiner_delta}), then by \thref{lemma_2.18} and the comment below
\[
\E \bigg|\sum_{i=1}^{l-k}h_{1,n}(Y_i)\bigg|^4 \leq C(l-k)^2\Delta_n^2,
\]
where $C$ does not depend on $l$, $k$ or $n$. Thus, 
\begin{equation*}
\p(|S_l-S_k|\geq \lambda) \leq \frac{1}{\lambda^4}\E|S_l-S_k|^4 \leq \frac{C(l-k)^2\Delta_n^2}{\lambda^4n^2} = \frac{1}{\lambda^4}\bigg(\sum_{i=k+1}^{l}u_{n,i}\bigg)^2,
\end{equation*}
where $u_{n,i} = C^{1/2} \Delta_n n^{-1}$. Hence, $S_j$ satisfies assumption (\ref{condition_prop_billingsley}) of \thref{Billingsley} with $\beta=4$, $\alpha=2$. Therefore, by (\ref{result_prop_billingsley}), for any fixed $\epsilon >0$, as $n\rightarrow\infty$,
\begin{align*}
\p\bigg(\max_{1\leq k \leq n}n^{-1/2} \Big|S_k^{(1)}\Big| \geq \epsilon\bigg) \leq \frac{K}{\epsilon^4}\bigg(\sum_{i=1}^{n} u_{n,i} \bigg)^2 = \frac{KC\Delta_n^2}{\epsilon^4} \rightarrow 0,
\end{align*}
since $\Delta_n\rightarrow 0$. The proof of (\ref{abschaetzung_h1}) for $i=2$ follows using a similar argument as in the proof for $i=1$.
\end{proof}

\begin{prop}\thlabel{Lemma_abschaetzung_bp}
Assume that $\left(Y_j\right)$ is $L_1$ near epoch dependent process on some absolutely regular process $\left(Z_j\right)$ and (\ref{condition_appr.const_regu.coeff}) holds. Furthermore, let Assumption \ref{assumption_change} be satisfied. Then
\begin{equation} \label{Lemma_abschaetzung_1}
\max_{1\leq k \leq k^*}n^{-3/2}\bigg| U_n(k,k^*) -k\left(n-k^*\right)\Theta_{\Delta_n} \bigg| = o_P\left(1\right)
\end{equation}
and 
\begin{equation} \label{Lemma_abschaetzung_2}
\max_{k^*\leq k \leq n}n^{-3/2}\bigg| U_n(k^*,k) -k^*\left(n-k\right)\Theta_{\Delta_n}  \bigg| = o_P\left(1\right),
\end{equation}
where $\Theta_{\Delta_n}$ is the same as in (\ref{remark_theta_delta_n}).
\end{prop}

\begin{proof}
By the Hoeffding decomposition (\ref{hoeffding_decomposition_hn}),
\begin{equation*}
h_n\left(x,y\right) - \Theta_{\Delta_n} = h_{1,n}\left(x\right) + h_{2,n}\left(y\right) + g_n\left(x,y\right).
\end{equation*}
Hence, 
\begin{multline*}
U_n(k,k^*)-k(n-k^*)\Theta_{\Delta_n} = \sum_{i=1}^{k}  \sum_{j=k^*+1}^{n} \big(h_{1,n}\left(Y_i\right) + h_{2,n}\left(Y_j\right) + g_n\left(Y_i,Y_j\right)\big)\\
= (n-k^*)\sum_{i=1}^{k}h_{1,n}\left(Y_i\right) + k \sum_{j=k^*+1}^{n}h_{2,n}\left(Y_j\right) + \sum_{i=1}^{k}  \sum_{j=k^*+1}^{n}g_n\left(Y_i,Y_j\right).
\end{multline*}
Denote
\begin{align}
U_n^{(g)}(k,k^*)&=\sum_{i=1}^{k}  \sum_{j=k^*+1}^{n}g_n\left(Y_i,Y_j\right),  &U_n^{(g)}(k^*,k)&=\sum_{i=1}^{k^*}  \sum_{j=k+1}^{n}g_n\left(Y_i,Y_j\right).\label{U_n_g}
\end{align}
Since $|n-k^*|\leq n$, $k^*\leq n$ and $\sum_{i=k^*+1}^{n}h_{2,n}(Y_j)=S_n^{(2)}-S_{k^*}^{(2)}$, then
\begin{align*}
\Big|U_n(k,k^*)-k(n-k^*)\Theta_{\Delta_n}\Big| &\leq n\Big(\big|S_k^{(1)}\big|+\big|S_{n}^{(2)}\big|+\big|S_{k^*}^{(2)}\big|\Big)+\Big|U_n^{(g)}(k,k^*)\Big|,\\
\Big|U_n(k^*,k)-k^*(n-k)\Theta_{\Delta_n}\Big| &\leq n\Big(\big|S_{k^*}^{(1)}\big|+\big|S_{n}^{(2)}\big|+\big|S_{k}^{(2)}\big|\Big)+\Big|U_n^{(g)}(k^*,k)\Big|,
\end{align*}
where $S_{k}^{(i)}$, $i=1,2$ are defined in (\ref{s_k}). Therefore,
\begin{multline}\label{abschaetzung_Lemma_abschaetzung_1}
\max_{1\leq k \leq k^*}n^{-3/2}\Big|U_n(k,k^*)-k(n-k^*)\Theta_{\Delta_n}\Big|\\
 \leq \max_{1\leq k \leq n}n^{-1/2}\Big(\big|S_k^{(1)}\big|+\big|S_{n}^{(2)}\big|+\big|S_{k^*}^{(2)}\big|\Big) + \max_{1\leq k \leq k^*}n^{-3/2}\Big|U_n^{(g)}(k,k^*)\Big|.
\end{multline}
The degenerate kernel $g_n$ is bounded and $1$-continuous, see Subsections \ref{Sub_U-statistics and Hoeffding decomposition} and \ref{Sub_1-continuity_h_hn}. Thus, by \thref{Proposition_DFGW} below, 
\begin{multline}\label{abschaetzung_gn_1}
\max_{1\leq k \leq k^*}n^{-3/2}\bigg| U_n^{(g)}(k,k^*)\bigg|\\
\leq \max_{1\leq k \leq n}n^{-3/2}\bigg|\sum_{i=1}^{k}\sum_{j=k+1}^{n}g_n(Y_i,Y_j)\bigg|+\max_{1\leq k \leq k^*}n^{-3/2}\bigg|\sum_{i=1}^{k}\sum_{j=k+1}^{k^*}g_n(Y_i,Y_j)\bigg|
 =o_P(1).
\end{multline}
Similar argument implies
\begin{equation*}
\max_{k^*< k \leq n}n^{-3/2}\bigg| U_n^{(g)}(k^*,k)\bigg| =o_P(1).
\end{equation*}
Using in (\ref{abschaetzung_Lemma_abschaetzung_1}) the bounds (\ref{abschaetzung_gn_1}) and (\ref{abschaetzung_h1}) of \thref{lemma_abschaetzung_h}  we obtain
\[
\max_{1\leq k \leq k^*}n^{-3/2}\Big|U_n(k,k^*)-k(n-k^*)\Theta_{\Delta_n}\Big| = o_P(1)
\]
which proves (\ref{Lemma_abschaetzung_1}). The proof of (\ref{Lemma_abschaetzung_2}) follows using similar argument.
\end{proof}

Denote
\begin{equation}\label{tilde_U_n_g}
\tilde{U}_n^{(g)}(k) = \sum_{i=1}^{k}\sum_{j=k+1}^{n}g(Y_i,Y_j).
\end{equation}

\begin{lemma}\thlabel{hajek_renyi_g}
Assume that $\left(Y_j\right)$ is $L_1$ near epoch dependent process on some absolutely regular process $\left(Z_j\right)$ and (\ref{condition_appr.const_regu.coeff}) holds. Furthermore, let Assumption \ref{assumption_change} be satisfied and let $a(n) = M/\Delta_n^2$, $M>0$, and  $\tilde{U}_n^{(g)}(k)$,  $U_n^{(g)}(k,k^*)$ and $U_n^{(g)}(k^*,k)$ are defined as in (\ref{tilde_U_n_g}), (\ref{U_n_g}). Then there exists $C>0$ such that for any $\epsilon>0$,
\begin{equation}\label{hajek_renyi_g_1}
\p\bigg(\max_{k:|k-k^*|\geq a(n)}\Big|\frac{\tilde{U}_n^{(g)}(k)}{k^*-k}\Big|>\epsilon\bigg) \leq \frac{C}{\epsilon^2}\Big(\frac{n^2}{a(n)}+\frac{1}{n}\Big),
\end{equation}
\begin{equation}\label{hajek_renyi_g_2}
\p\bigg(\max_{1\leq k \leq k^*-a(n)}\Big|\frac{U_n^{(g)}(k,k^*)}{k^*-k}\Big|>\epsilon\bigg) \leq \frac{C}{\epsilon^2}\Big(\frac{n^2}{a(n)}+\frac{1}{n}\Big),
\end{equation}
\begin{equation*}
\p\bigg(\max_{k^*+a(n)\leq k \leq n}\Big|\frac{U_n^{(g)}(k^*,k)}{k-k^*}\Big|>\epsilon\bigg) \leq \frac{C}{\epsilon^2}\Big(\frac{n^2}{a(n)}+\frac{1}{n}\Big),
\end{equation*}
where $C$ does not depend on $\epsilon$, $n$ and $a(n)$.
\end{lemma}
\par 
\begin{proof}
Recall $\{k:|k-k^*|\geq a(n)\}=\{k\leq k^*-a(n)\}\cup\{k\geq k^*+a(n)\}$. We consider only the case $\max_{1\leq k \leq k^*-a(n)}$ since the proof for $\max_{k^*+a(n)\leq k\leq n}$ is similar.\par

\textit{Proof of (\ref{hajek_renyi_g_1}).} Define $R_k = \tilde{U}_n^{(g)}(k)-\tilde{U}_n^{(g)}(k-1)$, $k\geq 1$, $\tilde{U}_n^{(g)}(0)=0$ and $R_0=0$. Then $\tilde{U}_n^{(g)}(k)=\sum_{i=1}^k R_i$. Inequality (\ref{inequality_hr}) of \thref{theorem_KL}, applied to the random variables $R_i$ with $c_k=1/(k^*-k)$ yields
\begin{multline}\label{zz_hajek_renyi_g}
\rho_n := \epsilon^2 \p\bigg(\max_{1\leq k \leq k^*-a(n)} \frac{1}{k^*-k}\bigg|\sum_{i=1}^{k}R_i\bigg| >\epsilon\bigg)\\
 \leq \frac{1}{(k^*-1)^2}\E R_1^2 + \sum_{k=1}^{k^*-a(n)}\bigg|\frac{1}{(k^*-k-1)^2}-\frac{1}{(k^*-k)^2}\bigg|\E\bigg(\sum_{i=1}^{k}R_i\bigg)^2\\
+ 2\sum_{k=1}^{k^*-a(n)}\frac{1}{(k^*-k-1)^2}\E\bigg(\left|R_{k+1}\right|\bigg|\sum_{j=1}^{k}R_j\bigg|\bigg) + \sum_{k=1}^{k^*-a(n)}\frac{1}{(k^*-k-1)^2}\E  R_{k+1}^2.
\end{multline}
In Subsections \ref{Sub_U-statistics and Hoeffding decomposition} and \ref{Sub_1-continuity_h_hn}, we showed that kernel function $g(x,y)$ is bounded and $1$-continuous. Therefore, by \thref{Lemma_DFGW} below
\begin{equation}\label{hr_g_1}
\E\Big[ (\tilde{U}_n^{(g)}(k))^2\Big]=\E\bigg(\sum_{i=1}^{k}R_i\bigg)^2 \leq Ck(n-k), \qquad k=1,\ldots,n.
\end{equation}
\thref{Lemma_DFGW} also yields
\begin{equation}\label{hr_g_2}
\E R_{k+1}^2 = \E\left(\tilde{U}_n^{(g)}(k+1)-\tilde{U}_n^{(g)}(k)\right)^2 \leq n^3C\frac{(k+1)-k}{n^2} = Cn, \qquad k=1,\ldots,n.
\end{equation}
Then,
\begin{equation*}
\E\bigg(\left|R_{k+1}\right|\bigg|\sum_{j=1}^{k}R_j\bigg|\bigg) \leq \bigg(\E R_{k+1}^2 \bigg)^{1/2}\bigg(\E \Big(\sum_{j=1}^{k}R_j\Big)^2 \bigg)^{1/2}\leq C\sqrt{n}\sqrt{k(n-k)}.
\end{equation*}
From (\ref{zz_hajek_renyi_g}), (\ref{hr_g_1}) and (\ref{hr_g_2}), using $\frac{1}{(k^*-k-1)^2}-\frac{1}{(k^*-k)^2}\leq \frac{2}{(k^*-k-1)^3}$, we obtain
\begin{equation*}
\rho_n \leq C\bigg[ \frac{n-1}{(k^*-1)^2} + \sum_{k=1}^{k^*-a(n)}\bigg\{\frac{k(n-k)}{(k^*-k-1)^3}
+ \frac{\sqrt{n}\sqrt{k(n-k)}+n}{(k^*-k-1)^2}\bigg\}\bigg] .
\end{equation*}
Noting that $\sqrt{k(n-k)}\leq n$, $(k^*-k-1)^{-3}\leq (k^*-k-1)^{-2}$, it follows 
\begin{equation*}
\rho_n \leq C\bigg( \frac{1}{n}+ \sum_{k=1}^{k^*-a(n)}\frac{n^2}{(k^*-k-1)^2}\bigg) \leq C\Big(\frac{1}{n}+\frac{n^2}{a(n)}\Big).
\end{equation*}
\textit{Proof of (\ref{hajek_renyi_g_2}).} It follows a similar line to the proof of (\ref{hajek_renyi_g_1}). Denote $\tilde{R}_k=U_n^{(g)}(k,k^*)-U_n^{(g)}(k-1,k^*)$. We verified in Subsections \ref{Sub_U-statistics and Hoeffding decomposition} and \ref{Sub_1-continuity_h_hn} that function $g_n(x,y)$ is bounded and $1$-continuous. Therefore, by \thref{Lemma_DFGW} below,
\begin{equation*}
\E \Big[(U_n^{(g)}(k,k^*))^2\Big]=\E\bigg(\sum_{i=1}^{k}\tilde{R}_i\bigg)^2 \leq Ck(n-k^*), \qquad k=1,\ldots,k^*
\end{equation*}
and 
\begin{multline*}
\E \tilde{R}_{k+1}^2 = \E\left(U_n^{(g)}(k+1,k^*)-U_n^{(g)}(k,k^*)\right)^2\\ 
= \E\Big(\sum_{j=k^*+1}^{n}g_n(Y_{k+1},Y_j)\Big)^2 \leq C(n-k^*), \qquad k=1,\ldots,k^*.
\end{multline*}
Combining both bounds, we obtain
\begin{equation*}
\E\bigg(\left|\tilde{R}_{k+1}\right|\bigg|\sum_{j=1}^{k}\tilde{R}_j\bigg|\bigg) \leq \bigg(\E \tilde{R}_{k+1}^2 \bigg)^{1/2}\bigg(\E \Big(\sum_{j=1}^{k}\tilde{R}_j\Big)^2 \bigg)^{1/2}\leq C(n-k^*)\sqrt{k}.
\end{equation*}
Using the same argument as in the proof of (\ref{hajek_renyi_g_1}), we obtain
\[
\p\bigg(\max_{1\leq k \leq k^*-a(n)}\frac{1}{k^*-k}\Big|U_n^{(g)}(k,k^*)\Big|>\epsilon\bigg) \leq \frac{C}{\epsilon^2}\Big(\frac{1}{n}+\frac{n^2}{a(n)}\Big).
\]
This completes proof of (\ref{hajek_renyi_g_2}) and the lemma.
\end{proof}

\begin{lemma}\thlabel{Lemma_abschaetzung_gewichtetes_maximum}
Assume that $\left(Y_j\right)$ is a stationary zero mean $L_1$ near epoch dependent process on some absolutely regular process $\left(Z_j\right)$ and (\ref{condition_appr.const_regu.coeff}) holds. Furthermore, let Assumption \ref{assumption_change} be satisfied and let $a(n) = M/\Delta_n^2$, $M>0$. Then, as $n\rightarrow\infty$, $M\rightarrow\infty$,\par \bigskip
\textit{(i)} For any $\epsilon>0$,
\begin{equation}\label{zz_gewichtetes_maximum}
\p\bigg(\frac{1}{n\Delta_n}\max_{k:|k-k^*|\geq a(n)}\bigg|\frac{U_n(k^*)-U_n(k)}{k^*-k}\bigg|>\epsilon\bigg)\rightarrow 0.
\end{equation}
\textit{(ii)} For any $\epsilon>0$,
\begin{equation}\label{zz_gewichtetes_maximum_unkk}
\p\bigg(\frac{1}{n\Delta_n}\max_{1\leq k \leq k^*-a(n)}\bigg|\frac{U_n(k^*,k^*)-U_n(k,k^*)}{k^*-k}-(n-k^*)\Theta_{\Delta_n}\bigg|>\epsilon\bigg)\rightarrow 0,
\end{equation}
and 
\begin{equation*}
\p\bigg(\frac{1}{n\Delta_n}\max_{k^*+a(n)\leq k \leq n}\bigg|\frac{U_n(k^*,k^*)-U_n(k^*,k)}{k-k^*}-k^*\Theta_{\Delta_n}\bigg|>\epsilon\bigg)\rightarrow 0,
\end{equation*}
where $\Theta_{\Delta_n}$ is the same as in (\ref{remark_theta_delta_n}). 

\end{lemma}

\begin{proof}
Notice that $\{k:|k-k^*|\geq a(n)\}=\{k\leq k^*-a(n)\}\cup\{k\geq k^*+a(n)\}$. We will prove relations (\ref{zz_gewichtetes_maximum}) and (\ref{zz_gewichtetes_maximum_unkk}) for $\max_{1\leq k \leq k^*-a(n)}$. The proof for $\max_{k^*+a(n)\leq k\leq n}$ is similar.\par 
Notice, that $a(n)=\frac{Mn}{\Delta_n^2n}=o(Mn)$ since $n\Delta_n^2\rightarrow\infty$ by assumption (\ref{assumption_size_change_infty}). Therefore, for a fixed $M$, $a(n)=o(k^*)$ and $k^*-a(n)>1$ as $n\rightarrow\infty$. \par 
\textit{(i)} Denote
\begin{equation*}
\tilde{S}_k = \sum_{i=1}^k h_1(Y_i).
\end{equation*}

By Hoeffding's decomposition (\ref{hoeffding_decomposition_h}), for $k\leq k^*$, and using $h_1(x)=-h_2(x)$, it follows
\begin{multline*}
U_n(k) = \sum_{i=1}^{k}\sum_{j=k+1}^{n} \big(h_1(Y_i)+ h_2(Y_j) + g(Y_i,Y_j)\big)\\
 = (n-k)\sum_{i=1}^{k}h_1(Y_i) - k \sum_{j=k+1}^{n}h_1(Y_j) + \sum_{i=1}^{k}\sum_{j=k+1}^{n}g(Y_i,Y_j) = n\tilde{S}_k- k\tilde{S}_n + \tilde{U}_n^{(g)}(k),
\end{multline*}
where $\tilde{U}_n^{(g)}(k)$ is defined in (\ref{tilde_U_n_g}).
Hence,
\begin{equation*}
\Big|U_n(k^*)-U_n(k)\Big| = \Big|n\big(\tilde{S}_{k^*}-\tilde{S}_k\big)-(k^*-k)\tilde{S}_n + \tilde{U}_n^{(g)}(k^*)-\tilde{U}_n^{(g)}(k)\Big|.
\end{equation*}
Therefore, for $1\leq k\leq k^*-a(n)$,
\begin{align*}
\frac{1}{n\Delta_n}\Big|\frac{U_n(k^*)-U_n(k)}{k^*-k}&\Big|\\
&\leq  \frac{1}{\Delta_n}\frac{\big|\tilde{S}_{k^*}-\tilde{S}_k\big|}{k^*-k}
 + \frac{1}{\Delta_n}\frac{\big|\tilde{S}_n\big|}{n} + \frac{\big|\tilde{U}_n^{(g)}(k^*)\big|}{n\Delta_na(n)} + \frac{1}{n\Delta_n}\frac{\big|\tilde{U}_n^{(g)}(k)\big|}{k^*-k}\\
 &=: \rho_k^{(1)}+\rho_k^{(2)}+\rho_k^{(3)}+\rho_k^{(4)}.
\end{align*}

It suffices to show that for any $\epsilon>0$, as $n\rightarrow\infty$, for $l=1,\ldots,4$,
\begin{equation}\label{zz_gewichtetes_maximum_2}
\p \Big( \max_{1\leq k \leq k^*-a(n)} \rho_k^{(l)} >\epsilon \Big) \rightarrow 0, \quad M\rightarrow \infty,
\end{equation}  
which proves (\ref{zz_gewichtetes_maximum}) for $\max_{1\leq k\leq k^*-a(n)}$.\par 
For $l=1$, stationarity of the process $(h_1(Y_j))$ yields
\begin{equation*}
\Big\{ \big|\tilde{S}_{k^*}-\tilde{S}_{k}\big|=\big| \sum_{i=k+1}^{k^*}h_1(Y_i) \big|,\; 1\leq k\leq k^*-a(n) \Big\} \overset{d}{=} \Big\{ \big|\tilde{S}_{k^*-k}\big|,\; 1\leq k\leq k^*-a(n) \Big\}.
\end{equation*}
Therefore,
\begin{equation*}
\max_{1\leq k\leq k^*-a(n)}\rho_k^{(1)} \overset{d}{=} \frac{1}{\Delta_n}\max_{k\geq 1:\; k^*-k\geq a(n)}\frac{\big|\tilde{S}_{k^*-k}\big|}{k^*-k} \overset{d}{=} \frac{1}{\Delta_n}\max_{a(n)\leq j\leq n}\frac{\big|\tilde{S}_{j}\big|}{j}.
\end{equation*}
Since $(h_1(Y_j))$ is $L_1$ NED on an absolutely regular process $(Z_j)$, $\E h_1(Y_1)=0$ and $|h_1(x)|\leq 1/2$, then by \thref{hajek_renyi_inequality},
\begin{equation}\label{hri}
\max_{a(n)\leq k \leq n}\frac{1}{k}\big|\tilde{S}_{k}\big| = O_P\bigg(\frac{1}{\sqrt{a(n)}}\bigg).
\end{equation}
Thus,
\[
\max_{1\leq k\leq k^*-a(n)}\rho_k^{(1)} = O_P\bigg(\frac{1}{\Delta_n\sqrt{a(n)}}\bigg) = O_P\bigg(\frac{1}{\sqrt{M}}\bigg) = o_P(1), \quad \text{as} \; M\rightarrow\infty,
\]
which proves (\ref{zz_gewichtetes_maximum_2}) for $l=1$.\par 
For $l=2$, by (\ref{hri}), $|\tilde{S}_n|/n = O_P(n^{-1/2})$. Thus, 
\[
\rho_k^{(2)} = \frac{1}{\Delta_n}\frac{\tilde{S}_n}{n} = O_P\bigg(\frac{1}{\Delta_n\sqrt{n}}\bigg) = o_P(1),
\]
since $\Delta_n\sqrt{n}\rightarrow\infty$ by (\ref{assumption_size_change_infty}), which proves (\ref{zz_gewichtetes_maximum_2}) for $l=2$.\par
To show (\ref{zz_gewichtetes_maximum_2}) for $l=3$, recall that $g(x,y)\leq 3/2$ is $1$-continuous, see Subsection \ref{Properties of the kernel}. Therefore, by \thref{Lemma_DFGW}, 
\[
\E\bigg( \frac{\tilde{U}_n^{(g)}(k^*)}{\sqrt{k^*(n-k^*)}} \bigg)^2 \leq C,
\]
which implies that
\[
\frac{\tilde{U}_n^{(g)}(k^*)}{\sqrt{k^*(n-k^*)}} = O_P(1).
\]
Thus,
\begin{multline*}
\rho_k^{(3)} = \frac{\big|\tilde{U}_n^{(g)}(k^*) \big|}{n\Delta_na(n)} = O_P\bigg( \frac{\sqrt{k^*(n-k^*)}}{n\Delta_na(n)} \bigg)\\ = O_P\bigg( \frac{1}{\Delta_na(n)} \bigg) = O_P\bigg( \frac{\Delta_n}{M} \bigg) = O_P\bigg( \frac{1}{M} \bigg) = o_P(1),\quad \text{as} \; M\rightarrow\infty,
\end{multline*}
which proves (\ref{zz_gewichtetes_maximum_2}) for $l=3$.\par 
Finally, for $l=4$, by \thref{hajek_renyi_g}, 
\begin{align*}
\p\Big( \max_{1\leq k \leq k^*-a(n)} \rho_k^{(4)} > \epsilon\Big) &= \p\Big( \max_{1\leq k \leq k^*-a(n)} \frac{\tilde{U}_n^{(g)}(k)}{k^*-k} > \epsilon n\Delta_n\Big)\\
&\leq \frac{C}{(\epsilon n\Delta_n)^2}\Big(\frac{n^2}{a(n)}+\frac{1}{n}\Big) = \frac{C}{\epsilon^2}\Big(\frac{1}{M}+\frac{1}{n^3\Delta_n^2}\Big) \rightarrow 0,\quad \text{as} \; M\rightarrow\infty, 
\end{align*}
which proves (\ref{zz_gewichtetes_maximum_2}) for $l=4$ and completes the proof of \textit{(i)}.\par \bigskip
\textit{(ii)} Let $S_k^{(1)}$, $S_k^{(2)}$ and $U_n^{(g)}(k,k^*)$ be defined as in (\ref{s_k}) and (\ref{U_n_g}). By Hoeffding's decomposition (\ref{hoeffding_decomposition_hn}), for $k\leq k^*$,
\begin{multline*}
U_n(k,k^*)-k(n-k^*)\Theta_{\Delta_n} = \sum_{i=1}^{k}\sum_{j=k^*+1}^{n} \big(h_{1,n}(Y_i)+ h_{2,n}(Y_j) + g_n(Y_i,Y_j)\big)\\
= (n-k^*)\sum_{i=1}^{k}h_{1,n}(Y_i) + k \sum_{j=k^*+1}^{n}h_{2,n}(Y_j) + \sum_{i=1}^{k}\sum_{j=k^*+1}^{n}g_n(Y_i,Y_j)\\
= (n-k^*)S_k^{(1)}+k(S_n^{(2)}-S_{k^*}^{(2)}) + U_n^{(g)}(k,k^*).
\end{multline*}
Hence, 
\begin{multline*}
\big|U_n(k^*,k^*)-U_n(k,k^*)-(k^*-k)(n-k^*)\Theta_{\Delta_n}\big|\\
=\big|(n-k^*)(S_{k^*}^{(1)}-S_k^{(1)})+(k^*-k)(S_n^{(2)}-S_{k^*}^{(2)}) + U_n^{(g)}(k^*,k^*)-U_n^{(g)}(k,k^*)\big|.
\end{multline*}
Therefore, for $1\leq k\leq k^*-a(n)$,
\begin{align*}
\frac{1}{n\Delta_n}&\Big|\frac{U_n(k^*,k^*)-U_n(k,k^*)}{k^*-k}-(n-k^*)\Theta_{\Delta_n}\Big|\\
&\leq  \frac{1}{\Delta_n}\frac{\big|S_{k^*}^{(1)}-S_k^{(1)}\big|}{k^*-k}
 + \frac{1}{\Delta_n}\frac{\big|S_n^{(2)}-S_{k^*}^{(2)}\big|}{n} + \frac{\big|U_n^{(g)}(k^*,k^*)\big|}{n\Delta_na(n)} + \frac{1}{n\Delta_n}\frac{\big|U_n^{(g)}(k,k^*)\big|}{k^*-k}\\
 &=: \nu_k^{(1)}+\nu_k^{(2)}+\nu_k^{(3)}+\nu_k^{(4)}.
\end{align*}
It suffices to show that for any $\epsilon>0$, as $n\rightarrow\infty$, for $l=1,\ldots,4$,
\begin{equation}\label{zz_gewichtetes_maximum_3}
\p \Big( \max_{1\leq k \leq k^*-a(n)} \nu_k^{(l)} >\epsilon \Big) \rightarrow 0, \quad M\rightarrow \infty,
\end{equation} 
which proves (\ref{zz_gewichtetes_maximum_unkk}) for $\max_{1\leq k \leq k^*-a(n)}$. The process $(h_{1,n}(Y_j))$ is stationary and $L_1$ NED on an absolutely regular process, see Section \ref{Properties of the kernel}. Furthermore, it has zero mean and $|h_{1,n}|\leq C\Delta_n$ by (\ref{abschaetzung_h1n_kleiner_delta}). Hence, by the same argument as for $\rho_k^{(1)}$, using \thref{hajek_renyi_inequality}, it follows
\begin{equation*}
\max_{1\leq k \leq k^*-a(n)}\nu_k^{(1)} \overset{d}{=} \frac{1}{\Delta_n}\max_{a(n)\leq j\leq n}\frac{|S_j^{(1)}|}{j}=O_P\Big(\frac{1}{\Delta_n\sqrt{a(n)}}\Big)=O_P\Big(\frac{1}{\sqrt{M}}\Big)=o_P(1),
\end{equation*}
as $M\rightarrow\infty$.\par 
\thref{lemma_abschaetzung_h} yields $\max_{1\leq k\leq n}n^{-1/2}|S_k^{(2)}|=o_P(1)$. Therefore,
\begin{equation*}
\max_{1\leq k \leq k^*-a(n)}\nu_k^{(2)} \leq \frac{1}{\sqrt{n}\Delta_n}\big(n^{-1/2}|S_n^{(2)}|+n^{-1/2}|S_{k^*}^{(2)}|\big)= o_P(1),
\end{equation*}
since $\sqrt{n}\Delta_n\rightarrow\infty$.\par 
We showed in Subsections \ref{Sub_U-statistics and Hoeffding decomposition} and \ref{Sub_1-continuity_h_hn} that the function $g_n(x,y)$ is bounded and $1$-continuous. Hence, by \thref{Lemma_DFGW},
\[
\E\bigg( \frac{{U}_n^{(g)}(k^*,k^*)}{\sqrt{k^*(n-k^*)}} \bigg)^2 \leq C.
\]
Therefore, the claim $\max_{1\leq k \leq k^*-a(n)}\nu_k^{(3)}=o_P(1)$ follows using the same argument as in the proof of (\ref{zz_gewichtetes_maximum_2}) for $l=3$.\par 
By \thref{hajek_renyi_g},
\begin{multline*}
\p\Big( \max_{1\leq k \leq k^*-a(n)} \nu_k^{(4)} > \epsilon\Big) = \p\Big( \max_{1\leq k \leq k^*-a(n)} \frac{|{U}_n^{(g)}(k,k^*)|}{k^*-k} > \epsilon n\Delta_n\Big)\\
\leq \frac{C}{(\epsilon n\Delta_n)^2}\Big(\frac{n^2}{a(n)}+\frac{1}{n}\Big) = \frac{C}{\epsilon^2}\Big(\frac{1}{M}+\frac{1}{n^3\Delta_n^2}\Big) \rightarrow 0,\quad \text{as} \; M\rightarrow\infty, 
\end{multline*}
which proves (\ref{zz_gewichtetes_maximum_3}) for $l=4$. This completes the proof of (\ref{zz_gewichtetes_maximum_unkk}) and the lemma.\par 
\end{proof}

\section{Auxiliary results from the literature} \label{result_literature}

This section contains results from the literature used in the proofs of this paper.\par 
\thref{lemma_2.18} states a correlation and a moment inequality for $L_1$ NED random variables, established by \citet{borovkova.2001}.

\begin{lemma}\thlabel{lemma_2.18}(Lemma 2.18 and 2.24, \citet{borovkova.2001}) 
Let $(Y_j)$ be $L_1$ near epoch dependent on an absolutely regular, stationary process with mixing coefficients $\beta_k$ and approximation constants $a_k$, and such that $|Y_0|\leq K\leq \infty$ a.s. Then, for all $i,k\geq 0$,
\[
\left| \Cov(Y_i,Y_{i+k})\right| \leq 4 Ka_{\lfloor \frac{k}{3} \rfloor} + 2 K^2 \beta_{\lfloor \frac{k}{3} \rfloor}.
\]
In addition, if $\sum_{k=0}^{\infty}k^2 (a_k+\beta_k)<\infty$, then there exists $C>0$ such that for all $n\geq 1$
\begin{equation}\label{lemma_2.24_result}
\E\bigg(\sum_{i=1}^n \big(Y_i-\E Y_i\big)\bigg)^4 \leq Cn^2.
\end{equation}
\end{lemma}

The proof of Lemma 2.24 in \citet{borovkova.2001} shows that (\ref{lemma_2.24_result}) holds with $C=C_0K^2$, where $C_0>0$ does not depend on $K$ and $n$.\par 
\bigskip

In Theorem 3 of \citet{Dehling.2015} the asymptotic distribution of the Wilcoxon test statistic for $L_1$ NED random process is  obtained. We use this result to show the consistency of the Wilcoxon-type estimator $\hat{k}$. 

\begin{theorem}\thlabel{Theorem_DFGW}(Theorem 3, \citet{Dehling.2015}) 
Assume that $\left(Y_j\right)$ is stationary and $L_1$ near epoch dependent process on some absolutely regular process $\left(Z_j\right)$ and (\ref{condition_appr.const_regu.coeff}) holds. Then, 
\begin{equation*}
\frac{1}{n^{3/2}}\max_{1\leq k < n}\bigg| \sum_{i=1}^k\sum_{j=k+1}^n \left(1_{\left\{ Y_i \leq Y_j \right\}}-1/2\right) \bigg| \overset{d}{\rightarrow} \sigma \sup_{0\leq \tau \leq 1}\left| B\left(\tau\right) \right|,
\end{equation*}
where $\left(B\left(\tau\right)\right)_{0\leq \tau \leq 1}$ is the standard Brownian bridge process,
\begin{equation*}
\sigma^2 = \sum_{k=-\infty}^\infty \Cov\left(F\left(Y_k\right),F\left(Y_0\right)\right),
\end{equation*}
and $F$ denotes the distribution function of $Y_j$.
\end{theorem}

We use the following results from \citet{Dehling.2015} to handle the degenerate part $g(x,y)$ of the Hoeffding decomposition (\ref{hoeffding_decomposition_h}).

\begin{prop}\thlabel{Proposition_DFGW}(Proposition 1, \citet{Dehling.2015}) 
Let $(Y_j)$ be stationary and $L_1$ near epoch dependent on an absolutely regular process with mixing coefficients $\beta_k$ and approximation constants $a_k$ satisfying
\begin{equation*}
\sum_{k=1}^{\infty}k\left(\beta_k+\sqrt{a_k}+\phi(a_k)\right) < \infty,
\end{equation*}
with $\phi(\epsilon)$ as in \thref{1-continuous}.
If $g(x,y)$ is a 1-continuous bounded degenerate kernel, then, as $n\rightarrow\infty$,
\[
\frac{1}{n^{3/2}}\max_{1\leq k\leq n}\bigg| \sum_{i=1}^{k}\sum_{j=k+1}^{n}g\left(Y_i,Y_j\right)\bigg|\rightarrow_p 0.
\]
\end{prop}

\begin{lemma}\thlabel{Lemma_DFGW}(Lemma 1 and 2, \citet{Dehling.2015}) 
Under assumptions of \thref{Proposition_DFGW} there exists $C>0$ such that for all $1\leq m \leq k \leq n$, $n\geq 2$,
\begin{equation*}
\E \bigg(\sum_{i=1}^{k}\sum_{j=k+1}^{n}g(Y_i,Y_j) \bigg)^2 \leq Ck(n-k),
\end{equation*}
\[
\E\bigg(n^{-3}\bigg|\sum_{i=1}^{k}\sum_{j=k+1}^{n}g(Y_i,Y_j)-\sum_{i=1}^{m}\sum_{j=m+1}^{n}g(Y_i,Y_j)\bigg|^2\bigg) \leq C\frac{k-m}{n^2}.
\]
\end{lemma}

In our proofs we use the maximal inequality of \citet{Billingsley.1999}, which is valid for stationary/non-stationary and independent/dependent random variables $\xi_i$.

\begin{theorem}\thlabel{Billingsley}(Theorem 10.2, \citet{Billingsley.1999})
Let $\xi_1,\ldots,\xi_n$ be random variables and $S_k= \sum_{i=1}^{k}\xi_k$, $k\geq 1$, $S_0=0$ denotes the partial sum. Suppose that there exist $\alpha > 1$, $\beta>0$ and non-negative numbers $u_{n,1},\ldots,u_{n,n}$ such that
\begin{equation}\label{condition_prop_billingsley}
\p\bigg(\left|S_j-S_i\right|\geq \lambda\bigg) \leq \frac{1}{\lambda^{\beta}}\bigg(\sum_{l=i+1}^{j}u_{n,l}\bigg)^{\alpha},
\end{equation}
for $\lambda>0$, $0\leq i \leq j \leq n$. Then for all $\lambda>0$, $n\geq 2$,
\begin{equation}\label{result_prop_billingsley}
\p\left(\max_{1\leq  k \leq n}\left|S_k\right| \geq \lambda\right) \leq \frac{K}{\lambda^{\beta}}\bigg(\sum_{l=1}^{n}u_{n,l}\bigg)^{\alpha},
\end{equation}
where $K>0$ depends only on $\alpha$ and $\beta$. 
\end{theorem}
By the Markov inequality, (\ref{condition_prop_billingsley}) is satisfied if
\begin{equation*}
\E\left|S_j-S_i\right|^{\beta} \leq \bigg(\sum_{l=i+1}^{j}u_{n,l}\bigg)^{\alpha}.
\end{equation*}

In the proof of  \thref{hajek_renyi_inequality} we use a H\'{a}jek-R\'{e}nyi type inequality established by \citet{Kokoszka.2000}.
\begin{theorem}\thlabel{theorem_KL}(Theorem 4.1, \citet{Kokoszka.2000})
 Let $X_1,..., X_n$ be any random variables with finite second moments and $c_1, ..., c_n$ be any non-negative constants. Then
\begin{align}
\epsilon^2\p\bigg(\max_{m\leq k\leq n}c_k\bigg|\sum_{i=1}^{k}X_i\bigg|>\epsilon\bigg) &\leq c_m^2\sum_{i,j=1}^m\E\left(X_iX_j\right) + \sum_{k=m}^{n-1}\left|c_{k+1}^2-c_k^2\right|\sum_{i,j=1}^k\E\left(X_iX_j\right)\nonumber\\
&+ 2\sum_{k=m}^{n-1}c_{k+1}^2\E\bigg(\left|X_{k+1}\right|\bigg|\sum_{j=1}^{k}X_j\bigg|\bigg) + \sum_{k=m}^{n-1}c_{k+1}^2\E X_{k+1}^2.\label{inequality_hr}
\end{align}
\end{theorem}

\section*{Acknowledgement}
The author would like to thank Herold Dehling, Liudas Giraitis and Isabel Garcia for valuable discussions. The research was supported by the Collaborative Research Centre 823 \emph{Statistical modelling of nonlinear dynamic processes} and the Konrad-Adenauer-Stiftung.

\

\footnotesize

\end{document}